\documentclass[12pt]{amsart}
\usepackage{geometry} % see geometry.pdf on how to lay out the page. There's lots.
\usepackage{amsmath,amsthm,amsfonts,amssymb}
\usepackage{graphics,xy}
\usepackage{url}
\usepackage{xcolor}
\usepackage{stmaryrd}

\usepackage{graphicx}
\usepackage{ulem}
\usepackage[utf8]{inputenc}
\usepackage{subfigure}
\newtheorem{lemma}{Lemma}[section]
\newtheorem{proposition}[lemma]{Proposition}
\newtheorem{theorem}[lemma]{Theorem}

\newtheorem{corollary}[lemma]{Corollary}
\newtheorem{question}[lemma]{Question}

\theoremstyle{definition}
\newtheorem{definition}[lemma]{Definition}
\newtheorem{remark}[lemma]{Remark}

\newcommand{\Z}{\mathbb{Z}}

\newcommand{\Q}{\mathbb{Q}}
\newcommand{\C}{\mathbb{C}}
\newcommand{\N}{\mathbb{N}}

\newcommand{\rt}{\tilde{\rho}}
\newcommand{\slC}{\mathrm{SL}_2(\mathbb{C})}

\title[An embedding of skein algebras from Dehn-Thurston coordinates]{An embedding of skein algebras of surfaces into quantum tori from Dehn-Thurston coordinates}
\author{Renaud Detcherry and Ramanujan Santharoubane}
\date{} % delete this line to display the current date

\address{Institut de Mathématiques de Bourgogne, UMR 5584 CNRS, Université Bourgogne Franche-Comté, F-2100 Dijon, France}
\email{renaud.detcherry@u-bourgogne.fr}

\address{Laboratoire de mathématique d’Orsay, UMR 8628 CNRS,
Bâtiment 307, Université Paris-Saclay, 
91405 ORSAY Cedex, FRANCE}
\email{ramanujan.santharoubane@universite-paris-saclay.fr}

\begin{document}

\begin{abstract}We construct embeddings of Kauffman bracket skein algebras of surfaces (either closed or with boundary) into localized quantum tori using the action of the skein algebra on the skein module of the handlebody. We use those embeddings to study representations of Kauffman skein algebras at roots of unity and get a new proof of Bonahon-Wong's unicity conjecture. Our method allows one to explicitly reconstruct the unique representation with fixed classical shadow, as long as the classical shadow is irreducible with image not conjuguate to the quaternion group.
\end{abstract}
\maketitle
\section{Introduction}
\label{sec:intro}
For $\Sigma$ a compact connected oriented surface and $G$ a Lie group, its character variety is
$$X(\Sigma,G)=\lbrace \rho:\pi_1(\Sigma)\longrightarrow G\rbrace \sslash G,$$
where $G$ acts by conjugation. One of the simplest and most intruiging character variety is obtained when $G=\mathrm{SL}_2(\C),$ as the character variety is then well-understood algebraically while connected to hyperbolic geometry, knot theory and $3$-dimensional topology, with many beautiful applications. Concretely speaking, $X(\Sigma,\mathrm{SL}_2(\C))$ is just an algebraic variety of dimension $6g-6+3n$ when $\Sigma$ is a surface with negative Euler characteristic of genus $g$ with $n$ boundary components, and the ring of regular functions on the character variety $\C[X(\Sigma,\mathrm{SL}_2(\C))]$ is just a commutative algebra.

Skein algebras and skein modules, introduced independently by Przytycki \cite{Prz91} and by Turaev\cite{Tur88}, give a quantization of character varieties. The skein module $S(M)$ of a compact oriented $3$-manifold $M$ is a $\Z[A^{\pm 1}]$-module which is the quotient of the free module spanned by isotopy classes of framed links in the interior of $M,$ modulo the famous Kauffman relations:
%\begin{center}
%\input{Kauffman_rel.pdf_tex}
%\end{center}
\begin{align*}
 \begin{minipage}{0.6in}\includegraphics[width=0.7in]{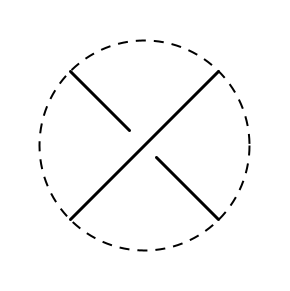}\end{minipage}
& =  
A
\begin{minipage}{0.7in}\includegraphics[width=0.7in]{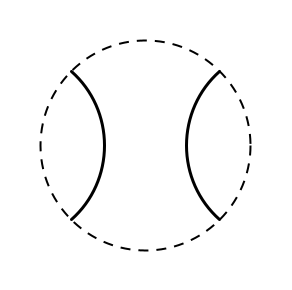} \end{minipage}
+ A^{-1} \begin{minipage}{0.7in} \includegraphics[width=0.7in]{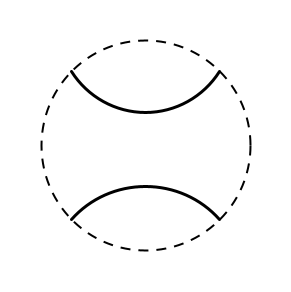}  \end{minipage} \\  
   L \cup U & = (-A^2-A^{-2}) L
   \end{align*}
where the first relation relates $3$ framed links in $M$ that are identical except in a small ball, and in the second relation $U$ is the unframed unknot. In the case where $M=\Sigma \times [0,1],$ with $\Sigma$ a compact oriented surface, the skein module has a natural structure of algebra given by the stacking operation. We will write $S(\Sigma)$ for the skein algebra of a surface (i.e. the skein module of $\Sigma\times [0,1],$ with its natural algebra structure). We recall also that when $M$ is a compact oriented surface, the skein module $S(M)$ is a module over the skein algebra $S(\partial M),$ again for the natural stacking operation.

A result of Bullock \cite{Bul97} and Przytycki--Sikora \cite{PS00} asserts that setting the parameter $A$ to $-1,$ the skein algebra $S(\Sigma)\otimes_{A=-1} \C$ is isomorphic to the commutative algebra $\C[X(\Sigma,\mathrm{SL}_2(\C))],$ and a further result of Turaev \cite{Tur91} states that $S(\Sigma)$ is actually a deformation quantization of that algebra in the direction of the Atiyah--Bott--Goldmann Poisson bracket. 

Moreover, the skein algebras of surfaces have deep ties with much of quantum topology, in particular with the Jones polynomials, the Witten--Reshetikhin--Turaev invariants of $3$-manifolds and their associated TQFTs. A better understanding of skein algebras (or their $3$-dimensional counterparts: skein modules) seems to be key for working on many of the open conjectures in quantum topology \cite{LZ17}\cite{MS21}\cite{BWY21}.
Contrary to functions on the $\mathrm{SL}_2(\C)$-character variety, skein algebras are a difficult object to tackle, in part due to being non-commutative algebras. For instance, except for a few small surfaces,  there is at the moment no-known  presentations of $S(\Sigma)$ besides the definition. 

However, pioneering work of Bonahon and Wong \cite{BW11}\cite{BW16}\cite{BW17}\cite{BW19} led to a breakthrough in our understanding of skein algebras. They constructed an embedding, the \textit{quantum trace map} from the skein algebras of a surface $\Sigma$ with $n\geq 1$ punctures to a quantum torus. A quantum torus is a non-commutative algebra of the form $\Z[A^{\pm 1}]\langle X_i, i \in I\rangle/\lbrace X_iX_j=A^{\sigma_{ij}}X_jX_i \rbrace;$ in a way, quantum tori are the simplest possible non-commutative algebras. Bonahon and Wong used the Chekov--Fock quantization of the Teichmüller space as their target space to define their quantum trace map, then they managed to quantize the map from character variety to shear coordinates. Some difficult computations are required to check that their formulas indeed yield an algebra morphism. Their method of defining a quantum trace map has since been simplified by L\^e in \cite{Le19}, and recently extended by L\^e and Yu to $\mathrm{SL}_n$ character variety.

The first result of this paper is to propose an alternative way of defining an embedding of the skein algebra into a quantum torus; or rather in our case a \textit{localized} quantum torus. Let $\mathcal{P}$ be a pants decomposition of $\Sigma$ consisting of $n$ curves non parallel to the boundary and $b$ curves parallel to the boundary. We write $\hat{\Sigma}$ for the closed surface obtained from $\Sigma$ by filling the boundary components by disks. We consider the quantum torus $\mathcal{T}(\mathcal{P})$ over $\Z[A^{\pm 1}]$ with $2n+b$ variables $E_1,\ldots,E_n,Q_1,\ldots,Q_n,C_1,\ldots,C_b$ where all variables commute except $Q_i$ and $E_i$ (for all $i \in \{1,\ldots,n\}$) that satisfy $Q_i E_i = A E_i Q_i$. Viewed as a ring, the quantum torus $\mathcal{T}(\mathcal{P})$ is an integral domain so we can define $\mathcal{A}(\mathcal{P})$ to be a $\Z[A^{\pm 1}]$-algebra containing $\mathcal{T}(\mathcal{P})$ where $A^k Q_i^2-A^{-k} Q_i^{-2}$ is invertible for all $1 \le i \le n$ and $k \in \Z$. The algebra $\mathcal{A}(\mathcal{P})$ will be called a localized quantum torus.

\begin{theorem}\label{thm:embeddingIntro}
There is an injective $\Z[A^{\pm 1}]$-algebra homomorphism $$\sigma : S(\Sigma) \to \mathcal{A}(\mathcal{P})$$
that factors through the natural action $S(\Sigma) \longrightarrow \mathrm{End}(S(H,\Q(A)))$ where $H$ is a handlebody with boundary $\hat{\Sigma},$ such that curves in $\mathcal{P}$ bound a disk in $H.$ Moreover for any curve $\alpha \in \mathcal{P}$ there exists $Q \in \{Q_1,\ldots,Q_n,C_1,\ldots,C_b \}$ such that $$\sigma(\alpha) = -(A^2 Q^2+A^{-2} Q^{-2})$$

\end{theorem}
We remark that when $\Sigma$ has boundary, $S(H,\Q(A))$ is the sum of relative skein modules of $H$ where we add a colored point in each filling disk of $\hat{\Sigma}.$

Note that in particular, we recover a theorem of L\^e \cite{Le21} about the faithfulness of the natural action of $S(\Sigma)$ on $S(H).$ 
The precise definition of the embedding will be given in Definition \ref{def:embedding}; it is based on the study of coefficients of \textit{curve operators} on $S(H,\Q(A))$ where $H$ is a handlebody with boundary $\Sigma,$ in some basis given by trivalent colored graphs (see Lemma \ref{lemma:basis}), which are the skein module version of the basis of WRT-TQFTs given by Blanchet--Habegger--Masbaum and Vogel \cite{BHMV}.

 Morally speaking our embedding could be thought as the quantization of Fenchel--Nielsen coordinates associated to a pair of pants decomposition, while Bonahon--Wong's quantum trace map is based on a quantization of shear coordinates on the Teichmüller space. 

 Compared with Bonahon--Wong's result our embedding has the drawback of landing in a localized quantum torus instead of just a quantum torus, but in exchange we get several nice features. First, our result applies to closed surfaces as well as surfaces with boundary, whereas Bonahon--Wong's embedding needs punctures to be defined, since they have to start with ideal triangulations of the surface $\Sigma.$ Secondly, our embedding arises in a more natural way, by studying the action of $S(\Sigma)$ on the skein module $S(H)$ of a handlebody with boundary $\Sigma,$ in the graph basis of $S(H)$ associated to the pair of pants decomposition $\mathcal{P}.$ The proof that we get an embedding does not require to check difficult formulas, and will be a simple by-product of the fact that $S(H)$ is a module over $S(\Sigma).$ 
 
Finally, a nice feature of this new embedding is that it is in a way almost surjective: the localized quantum torus can be viewed as a finite degree extension of the image of $\sigma$. Indeed we see in Theorem \ref{thm:embeddingIntro} that the elements $Q_1,\ldots,Q_n,C_1,\ldots,C_b$ satisfy a degree $4$ polynomial equation with coefficients in $\sigma(S(\Sigma))$. The next proposition says that a similar property holds for the elements $E_1,\ldots,E_n$.

\begin{proposition}\label{prop:finite_ext}
There exists a finite index subgroup $\Lambda$ of $\Z^n$ such that all $k=(k_1,\ldots,k_n) \in \Lambda$ : $$E_1^{k_1}\cdots E_n^{k_n} = \sum_{j \in I} \sigma(\gamma_j) G_j $$ where $\{\gamma_j \, \, j \in I \}$ is a finite set of multicurves on $\Sigma$ and $\{G_j \, , \, j \in I \}$ are rational fractions in $Q_1,\ldots,Q_n,C_1,\ldots,C_b$.
\end{proposition}

An other important aspect of Bonahon--Wong work is the study of irreducible finite dimensional complex representations of $S_{\xi}(\Sigma) = S(\Sigma) \otimes_{A=\xi} \C$ when $\xi$ is root of unity. More precisely, when $\xi$ is root of unity of order twice an odd number, Bonahon--Wong associate in \cite{BW16} to any such representation $\rho$ of $S_{\xi}(\Sigma)$ a canonical point $r_{\rho} \in X(\Sigma,\slC)$ called the classical shadow of $\rho$. A natural and important question is to know if points in $X(\Sigma,\slC)$ completely classify irreducible finite dimensional representations of $S_{\xi}(\Sigma)$. In \cite{BW17}, Bonahon--Wong proved that if $\Sigma$ has at least one boundary component, any point in $X(\Sigma,\slC)$ satisfying a geometric condition are the classical shadow of an irreducible representation of $S_\xi(\Sigma)$. When $\Sigma$ has no boundary component they removed this geometric condition to prove that $\rho \mapsto r_{\rho}$ is sujective in \cite{BW19}. They asked the question which points in $X(\Sigma,\slC)$ have a single preimage by the map $r_{\rho};$ the \textit{unicity conjecture} that they formulated was that for a open dense subset of $X(\Sigma,\slC),$ the classical shadow determines the representation. In this paper we prove:

\begin{theorem}\label{thm:lift_rep}
Suppose that $\Sigma$ has at most one boundary component. Let $\xi$ be a $2p$-th primitive root of unity with $p \ge 3$ an odd number and $\rho : S(\Sigma) \to \mathrm{End}(V)$ be an irreducible representation with classical shadow $r$. Let $\mathcal{P}$ be a pants decomposition of $\Sigma$ in the same orbit, under the action of the mapping class group of $\Sigma$, as the one shown in Figure \ref{fig_sausage}. Suppose that for $\alpha_1,\alpha_2,\alpha_3 \in \mathcal{P}$ bounding a pair of pants : \begin{equation} \label{eqG1}
2+\sum_{k=1}^3 \mathrm{Tr}(r(\alpha_k^2))- \prod_{k=1}^3 \mathrm{Tr}(r(\alpha_k)) \neq 0
\end{equation} and for all $\alpha \in \mathcal{P}$ 
\begin{equation} \label{eqG2}
\mathrm{Tr}(r(\alpha)) \neq \pm 2
\end{equation}Then there is a representation $\tilde{\rho} : \mathcal{A}_{\xi}(\Gamma)^0 \to \mathrm{End}(V)$ such that $\tilde{\rho} \circ \sigma_{\xi} = \rho.$ 
\end{theorem}

In this theorem $\mathcal{A}(\Gamma)^0$ is a $\Z[A^{\pm 1}]$-subalgebra of $\mathcal{A}(\mathcal{P})$ containing $\sigma(S(\Sigma))$. The definition of $\mathcal{A}(\Gamma)^0$ is given in Definition \ref{A0}, it is isomorphic to a localized quantum torus (see Lemma \ref{gen_t}). The notation $\mathcal{A}_{\xi}(\Gamma)^0$ stands for $\mathcal{A}(\Gamma)^0 \otimes_{A=\xi} \C$ and $\sigma_{\xi}$ is the induced map $S_{\xi}(\Sigma) \to \mathcal{A}_{\xi}(\Gamma)^0$. As representations of $\mathcal{A}_{\xi}(\Gamma)^0$ are well understood, we have the following corollary :

\begin{corollary} \label{cor:unicity}
Let $\rho_1$ and $\rho_2$ be two irreducible complex finite dimensional representations of $S_{\xi}(\Sigma)$ with same classical shadow $r$ satisfying the hypothesis of Theorem \ref{thm:lift_rep}. Moreover if $\Sigma$ has a boundary component, we suppose that $\rho_1$ and $\rho_2$ have the same scalar value on any simple closed curve parallel to the boundary. Then $\rho_1$ and $\rho_2$ are isomorphic with dimension $p^{3g-2}$ when $\Sigma$ has a boundary component and $p^{3g-3}$ when $\Sigma$ has no boundary (here $g$ is the genus of $\Sigma$).
\end{corollary}

Notice that the conditions satisfied by the classical shadow $r$ in Theorem \ref{thm:lift_rep} defined a Zariski open dense of $X(\Sigma,\slC)$. Therefore Corollary \ref{cor:unicity} recovers a theorem of Frohman, Kania-Bartoszynska and L\^e \cite{FKBL19} that showed that there is a Zariski open dense subset of $X(\Sigma,\slC)$ on which points have a single preimage. Our proof thus gives an alternative proof of the unicity conjecture of Bonahon and Wong. We note that the proof in \cite{FKBL19} used abstract arguments about Azumaya algebras and facts about the center of skein algebras at roots of unity, and produced an unexplicit Zariski open subset of $X(\Sigma,\slC)$ where the unicity conjecture holds.

An alternative (and more constructive) approach to the unicity conjecture had been initiated by Takenov \cite{Takenov}, in the case of where $\Sigma$ is the one-holed torus and the 4-holed sphere. It involved explicitely extending the representations of $S_{\xi}(\Sigma)$ to a quantum torus in which it embeds. The present paper will take advantage of the embedding defined by Theorem \ref{thm:embeddingIntro} to extend Takenov's strategy for compact oriented surfaces that are closed or have one boundary component. We show that a representation of classical shadow $r$ can be extended to a representation of a localized quantum torus and use an argument of Bonahon and Liu about the uniqueness of representations of quantum tori to deduce the unicity of the representation of $S_{\xi}(\Sigma)$ with classical shadow $r$. We note that this construction also allows one to reconstruct the representation from the classical shadow, recovering the surjectivity proved in \cite{BW19}, on the open dense subset of $X(\Sigma,\slC)$ described above.

The theorem of \cite{FKBL19} has since been improved by Ganev, Jordan and Safronov \cite{GJS19}: building upon Frohman--Kania-Bartoszynska--L\^e's result, they prove that the unicity conjecture actually holds over the set of all irreducible representations. In a forthcoming paper of the two authors with Thomas Le Fils, we prove:
\begin{theorem}\label{thm:compatible_pants}\cite{DLFS} Let $\Sigma$ be a closed surface. The set of conjugacy classes of representations $r:\pi_1(\Sigma) \to \slC$ satisfying the hypothesis of Theorem \ref{thm:lift_rep} is equal to the set of irreducible representations minus representations whose image is isomorphic to the quaternion group with $8$ elements.

\end{theorem}
This theorem together with Theorem \ref{thm:lift_rep} almost recovers the result of \cite{GJS19}. The fact that we have to exclude the representations with quaternionic image is due to the specific type of pair of pants decomposition we use in Theorem \ref{thm:lift_rep}. Indeed a representation with quaternionic image has trace $\pm 2$ on any separating closed curve on $\Sigma.$ We believe that the same arguments starting with a pair of pants decomposition using only non-separating curves would cover also the case of quaternionic irreducible representations. 

We end this introduction with the following question:
\begin{question} \label{question:integrality} Can one define an embedding similar to that of Theorem \ref{thm:embeddingIntro}, but with values in a quantum torus instead of a localized quantum torus ?
\end{question}
A possible approach towards this question would be to use basis for the skein module modeled on the integral basis of $\mathrm{SO}(3)$-TQFTs given by Gilmer and Masbaum \cite{GM07}. Note that it is however unlikely that one could both get an integral version of the embedding of Theorem \ref{thm:embeddingIntro}, while keeping the nice "finite extension" feature described in Proposition \ref{prop:finite_ext}. This last property is key for lifting representations as in Theorem \ref{thm:lift_rep}.

The paper is organized in two largely independent parts. Section \ref{sec:embedding} is devoted to the definition of the map $\sigma$ of Theorem \ref{thm:embeddingIntro}, and of the proof that it is an embedding. The short Section \ref{sec:sausage_case} introduces a special kind of pair of pants decompositions that are used in the next section and serves as a transition between the two parts of the paper. Section \ref{sec:representations} is devoted to the proof of Theorem \ref{thm:lift_rep} and Corollary \ref{cor:unicity}, and could in principle be read independently of Section \ref{sec:embedding}, although some formulas are deeply inspired by it. 

\textbf{Acknowledgements:} Over the course of this work, the first author was  partially supported by the project “AlMaRe” (ANR-19-CE40-0001-01) and by the
EIPHI Graduate School (ANR-17-EURE-0002). The authors thank Thang L\^e, Thomas Le Fils, Julien March\'e and Maxime Wolff for helpful conversations.

\section{Embeddings of skein algebras of surfaces into localized quantum tori}
\label{sec:embedding}
\subsection{Localized quantum tori}
\label{sec:qTorus}
Consider the $\Z[A^{\pm 1}]$-module 
$$\tilde{\mathcal{A}}_{n,b} = \Z[A^{ \pm 1}][E_1^{\pm 1},\ldots,E_{n}^{\pm 1}] \underset{\Z[A^{ \pm 1}]}{\otimes} \Q(A)(Q_1,\ldots,Q_{n},C_{1}, \ldots,C_{b})  $$
For convenience, we write an element of $\tilde{\mathcal{A}}_{n,b}$ as 
$$ \sum_{k \in \Z^n} \,  E_1^{k_1}\cdots E_n^{k_n} R_k $$ where $R_k \in  \Z[A^{\pm 1}](Q_1,\ldots,Q_{n},C_{1}, \ldots,C_{b}), \, \forall k \in \Z^n $ and all but finitely many $R_k$' are zero. Also $E^k$ denotes $E_1^{k_1}\cdots E_{n}^{k_{n}}$.

For $P\in \Q(A)(Q_1,\ldots,Q_n,C_{1}, \ldots,C_{b})$ and $k \in \Z^n$
 we define $$\hat{P}^{(k)} = P(A^{k_1}Q_1,\ldots,A^{k_n}Q_n,C_1,\ldots,C_b).$$  
Now we set the following multiplication on $\tilde{\mathcal{A}}_{n,b}$ 
$$\Big(\sum_{k \in \Z^n} \,  E^k R_k \Big) \Big(\sum_{l \in \Z^n} \,  E^l S_l \Big) = \sum_{k,l \in \Z^n} \,  E^{k+l} \hat{R}_k^{(l)} S_l $$ This multiplication makes $\tilde{\mathcal{A}}_{n,b}$ a non-commutative algebra over the ring $\Z[A^{\pm 1}]$. 

Let $\mathcal{R}$ be the subring  of $\Q(A)(Q_1,\ldots ,Q_n,C_1,\ldots ,C_b)$ consisting of all elements of the form $\dfrac{U}{V}$ where $U \in \Z[A^{\pm 1}][Q_1^{\pm 1},\ldots Q_n^{\pm 1},C_1^{\pm 1},\ldots,C_b^{\pm 1}]$ and $V$ is a finite product (possibily empty) of elements of the form $A^m Q_j^2-A^{-m} Q_j^{-2}$ for $m \in \Z$ and $1 \le j \le n$.

Let $\mathcal{A}_{n,b}$ be the sub-module of $\tilde{\mathcal{A}}_{n,b}$ generated by elements of the form $E^k R$ where $k \in \Z^n$ and $R \in \mathcal{R}$. It is clear from the multiplicative structure of $\tilde{\mathcal{A}}_{n,b}$ that $\mathcal{A}_{n,b}$ is a $\Z[A^{\pm 1 } ]$-subalgebra of $\tilde{\mathcal{A}}_{n,b}$. The non-commutative algebra $\mathcal{A}_{n,b}$ will be called the \textbf{localized quantum torus}.

\subsection{An embedding of skein algebras through curve operators}
\label{sec:embeddingCurveOp}
Let $\Gamma \subset \mathbb{S}^3$ be a planar banded uni-trivalent graph, we denote by $\mathcal{E}$ its set of edges. Let $\mathcal{U} \subset \mathcal{E}$ be the set of univalent edges of $\mathcal{E}$ and $\mathcal{E}' = \mathcal{E} \backslash \mathcal{U}$. Let $n$ be the total number of edges of $\Gamma$ joining two trivalent vertices and $b$ be the the number of univalent edges of $\Gamma$. We number edges in $\mathcal{E}'$ from $1$ to $n$ and edges in $\mathcal{U}$ from $n+1$ to $n+b$. We will also write $\mathcal{A}(\Gamma)$ for the localized quantum torus $\mathcal{A}_{n,b}$.

Finally, let $P$ be the set of triples $(e,f,g)\in \mathcal{E}^3$ such that the corresponding edges are adjacent to the same trivalent vertex.

Let $H \subset \mathbb{S}^3$ be a tubular neighborhood of $\Gamma$. We denote by $\hat{\Sigma}$ the boundary of $H$, we suppose that the genus of $\hat{\Sigma}$ is at least $1$. The univalent vertices of $\Gamma$ define banded points on $\hat{\Sigma}$ : $x_1,\ldots,x_b$. Let $\Sigma$ be the surface obtained from $\hat{\Sigma}$ by removing small open disks around each $x_j$. We suppose that $\Sigma$ has negative Euler characteristic.

For $c_1,\ldots,c_b$ a coloring of the banded points $x_1,\ldots,x_b,$ the relative skein module $S(H,\Q(A),c)$ will be the module generated by tangles with $c_i$ boundary points on the banded point $x_i$ and with the Jones-Wenzl idempotent $f_{c_i}$ inserted, modulo the Kauffman relations. Moreover the relative skein module $S(H,\Q(A))$ will be the direct sum of all $S(H,\Q(A),c)$ over all possible colorings $c$ of the banded points $x_1,\ldots,x_b.$
We denote by $S(\Sigma)$ the skein algebra of $\Sigma$ over $\Z[A^{\pm 1}]$. The algebra $S(\Sigma)$ acts on $S(H,\Q(A))$ by the stacking operation. For $\gamma$ a multicurve, we denote by $T^{\gamma}$ the action of $\gamma$ on $S(H,\Q(A))$.

A map $c: \mathcal{E} \to \N$ is called an admissible coloring if $\forall (e,f,g) \in P$, we have triangular inequalities $c(e)\leq c(f)+c(g)$ and $c(e)+c(f)+c(g)$ is even. We also introduce a lattice $\Lambda \subset \Z^{\mathcal{E}}$ by
$$\Lambda=\lbrace k\in \Z^{\mathcal{E}} \ | \ \forall (e,f,g)\in P, \ k(e)+k(f)+k(g)\in 2\Z   \rbrace.$$
Given an admissible coloring $c : \mathcal{E} \to \N$ of $\Gamma$, we denote by $\varphi_c \in S(H,\Q(A))$ the vector obtained by cabling each edge $e$ of $\Gamma$ using the Jones-Wenzl indempotents $f_{c(e)}.$ It is well known that 

\begin{lemma}\label{lemma:basis}
$\{ \varphi_c \, | \, c : \mathcal{E} \to \N \, \, \text{admissible} \}$ is a basis of $S(H,\Q(A))$.
\end{lemma}
\begin{proof}
The handlebody $H$ with banded points $x_1,\ldots,x_b$ is homeomorphic to the thickened surface $\Gamma \times [0,1],$ with banded points corresponding the univalent vertices. Fix a coloring $\hat{c}$ of the banded points. For thickened surfaces, the skein module is generated by disjoint unions of arcs and non-trivial simple closed curves with boundary $\hat{c}_i$ points on the $i$-th banded point and the $\hat{c}_i$-th Jones-Wenzl idempotent inserted at that banded point. Such tangles are completely determined by their intersection number with the co-core of each internal edge of $\Gamma.$ Let $\psi_c$ be the basis element which has $c_e$ intersections with the co-core of the edge $e.$ Now consider the vectors $\varphi_c$ corresponding to  admissible colorings of $\Gamma$ that coincides with $\hat{c_i}$ on the boundary. The recursive formula for the Jones-Wenzl idempotents shows that $\varphi_c$ is a linear combination of the vectors $\psi_d$ with with $d_e\leq c_e$ for all $e \in \mathcal{E'}.$ Moreover, $\varphi_c$ has nonzero coefficient along $\psi_c.$ This implies that the $\varphi_c$ are linearly independent and moreover an easy induction shows that any $\psi_c$ is a linear combination of the $\varphi_c,$ therefore the $\varphi_c$ are also a basis of $S(\Sigma,\Q(A)).$  
\end{proof}
The action of curve operators $T^{\gamma}$ in the basis $\varphi_c$ can be computed using the so-called fusion rules derived in \cite{MV94}. A complete set of fusion rules is described in Figure \ref{fig:fusion}, where coefficients are expressed in terms of $A$ and quantum integers $\lbrace n \rbrace=A^{2n}-A^{-2n}.$ 
\begin{figure}[!h]
\centering
\def \svgwidth{.6\columnwidth}
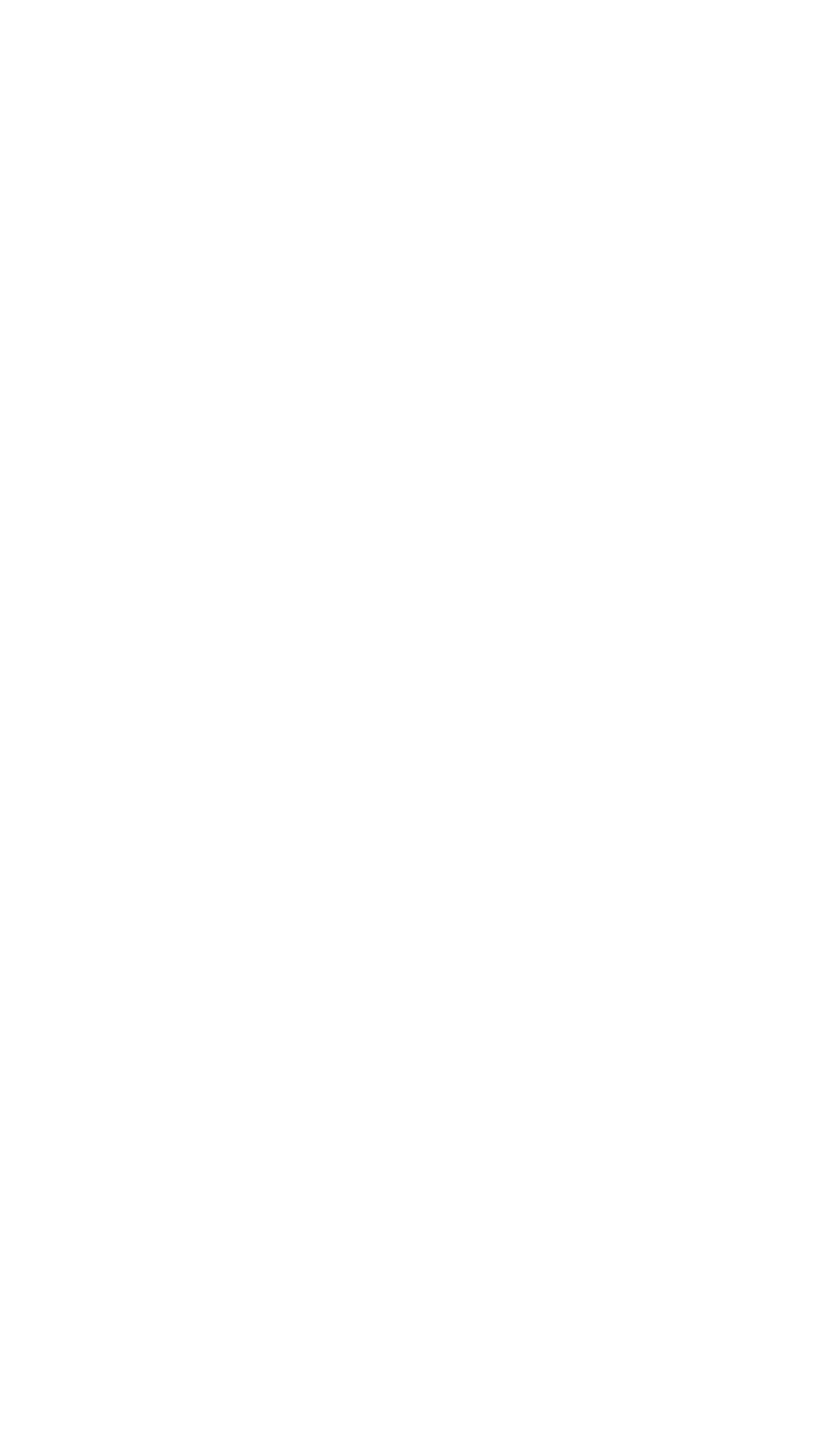
\caption{Fusion rules for computing curve operators  in the basis $\varphi_c.$ Thick edges represent edges of the trivalent graph $\Gamma,$ which are colored by integers, while dashed arcs are colored by 1. We let $\lbrace n \rbrace=A^{2n}-A^{-2n}.$}
\label{fig:fusion}
\end{figure}

Before studying the form of curve operators in the basis $\varphi_c,$ we will need the following lemma and definition:

\begin{lemma}\label{lemma:largeset}
Let $\Delta$ be the set of admissible colorings of $\Gamma$, then for all $v_1, \ldots,v_n \in \Lambda$ we have 
$$\bigcap_{j=1}^n \, \Delta + v_j \, \neq \emptyset.$$

\end{lemma}
We will call any subset of $\N^{\mathcal{E}}$ containing a subset of the form $\bigcap_{j=1}^n \, \Delta + v_j$ where $v_j\in \Lambda,$ a \textit{large subset} of $\N^{\mathcal{E}}.$
\begin{proof}
For any $w\in \Z^{\mathcal{E}}$ and $r\in \N$ let us write $B(w,r)$ be the ball around $w$ for $||\cdot||_{\infty}.$ Notice that for any $k \in \N,$ the set $\Delta$ contains $B((2k,\ldots,2k), 2\lfloor \frac{k}{3}\rfloor)\cap \Lambda.$ Let $r\geq 1,$ then if $2\lfloor \frac{k}{3} \rfloor>\mathrm{max}(||v_1||_{\infty},\ldots,||v_n||_{\infty})+r,$ then  $\bigcap_{j=1}^n \, \Delta + v_j$ contains $B((2k,\ldots,2k), r)\cap \Lambda.$ As $r$ can be chosen arbitrary large, this shows $\bigcap_{j=1}^n \, \Delta+v_j\neq \emptyset.$
\end{proof}

We can now describe the structure of curve operators $T^{\gamma}$ associated to multicurves on $\Sigma:$
\begin{proposition} \label{prop:expand}
Let $\gamma$ be a multicurve on $\Sigma$. There exists a large subset $V_{\gamma} \subset \N^{\mathcal{E}}$ and $F_k^{\gamma} \in \mathcal{R}$ (for $k : \mathcal{E}' \to \Z$) such that for all $c \in  V_{\gamma}$ 

$$ T^{\gamma} \varphi_c = \sum_{k : \mathcal{E}' \to \Z} F_k^{\gamma}(A^{c(1)},\ldots,A^{c(n)},A^{c(n+1)},\ldots,A^{c(n+b)}) \varphi_{c+k}$$ 
Moreover if $\mathcal{P} = \{\alpha_1,\ldots,\alpha_n\}$ is a pants decomposition of $\Sigma$ dual to $\Gamma$ (with numbering corresponding to the one of $\mathcal{E}$) then 
\begin{enumerate}
\item $F_k^{\gamma} = 0$ when $|k(j)| > i(\gamma,\alpha_j)$ or $k(j)\not\equiv i(\gamma,\alpha_j) \ (\mathrm{mod} \ 2)$ for some $1 \le j \le n$.
\item If $k = (\pm i(\gamma,\alpha_1),\ldots,\pm i(\gamma,\alpha_{n}))$ then $F_k^{\gamma} \neq 0$.
\end{enumerate}
\end{proposition}

The proof of Proposition \ref{prop:expand} involves fusion calculations to compute coefficients of curve operators, and will be done in Section \ref{sec:curveOpCoeff}.
\begin{lemma}\label{lemma:uniqueCoeff} Assume that for some multicurve $\gamma$ on $\Sigma,$ there exists two large subsets $V$ and $V'$ of $\N^{\mathcal{E}}$ and coefficients $F_k$ and $G_k \in \mathcal{R}$ such that for any $c \in V$ (resp. $V'$) $T^{\gamma}\varphi_c$ admits a decomposition as in Proposition \ref{prop:expand}. Then $F_k=G_k$ for all $k:\mathcal{E}'\rightarrow \Z.$
\end{lemma}
\begin{proof}
Indeed, the intersection $V\cap V'$ will be also be a large subset of $\N^{\mathcal{E}}$ and thus will contain by the proof of Lemma \ref{lemma:largeset} subsets of the form $B(v,r)\cap \Lambda$ where $v\in \N^{\mathcal{E}}$ and $r$ can be arbitrary large. Notice that $\Lambda$ contains the lattice $2\Z^{\mathcal{E}}.$ Assume that $r$ is strictly larger than  $d,$ the maximum of the degrees of the rational fractions $F_k-G_k$ (which we define as the maximum of the degree of their numerator and denominator). By Proposition \ref{prop:expand} the rational fractions $F_k$ and $G_k$ coincide on the set of all $(A^{c(1)},\ldots A^{c(n+b)})$ where $c \in B(v,r)\cap 2\Z^{\mathcal{E}},$ which is a product of sets that contains more than $d$ elements. By an easy induction on the number of variables $n$ of $F_k$ and $G_k,$ we can deduce that $F_k=G_k.$
\end{proof}
Thanks to Lemma \ref{lemma:uniqueCoeff}, we can make the following definition:
\begin{definition}\label{def:embedding}
For $\gamma$ a multicurve, we define $$\sigma(\gamma) = \sum_{k : \mathcal{E}' \to \Z} E^k F_k^{\gamma}(Q_1,\ldots,Q_{n},C_1,\ldots,C_b) \in \mathcal{A}(\Gamma)$$ and we extend linearly this definition to a $\Z[A^{\pm 1}]$-module morphism $$\sigma : S(\Sigma) \to \mathcal{A}(\Gamma)$$
\end{definition}

\begin{lemma}
$\sigma : S(\Sigma) \to \mathcal{A}(\Gamma)$ is a $\Z[A^{\pm 1}]$-algebra morphism.
\end{lemma}
\begin{proof}
Note that the map $\gamma \in S(\Sigma,\Z[A^{\pm 1}])\longmapsto T^{\gamma}\in \mathrm{End}(S(H,\Q(A)))$ is a morphism of algebra.
The lemma will follow from the fact that $\sigma(\gamma)$ encodes the action of $T^{\gamma} \in \mathrm{End}(S(H,\Q(A)))$ in the basis $\varphi_c$ and that the multiplication in $\mathcal{A}(\Gamma)$ corresponds to the composition of operators.

Indeed, let $\gamma$ and $\delta$ be two multicurves, and assume that 
$$T^{\gamma}\varphi_c=\underset{k:\mathcal{E}'\rightarrow \Z}{\sum}F_k^{\gamma}(A^{c(1)},\ldots,A^{c(n+b)})\varphi_{c+k}$$
for any $c\in V_{\gamma}$ and 
$$T^{\delta}\varphi_c=\underset{k:\mathcal{E}'\rightarrow \Z}{\sum}F_k^{\delta}(A^{c(1)},\ldots,A^{c(n+b)})\varphi_{c+k}$$
for any $c\in V_{\delta}.$ Let $V=\underset{k\in \Lambda, |k_i|\leq i(\gamma,\alpha_i)}{\bigcap} V_{\delta}+k.$ Then $V$ is a large subset of $\N^{\mathcal{E}'}$ and for any $c\in V$ and any $k\in \Lambda$ such that $F_k^{\gamma} \neq 0,$ we have $c+k\in V_{\delta}.$ Hence we get
\begin{multline*}T^{\delta\cdot\gamma}\varphi_c=T^{\delta}\left(\underset{k:\mathcal{E}'\rightarrow \Z}{\sum}F_k^{\gamma}(A^{c(1)},\ldots,A^{c(n+b)})\varphi_{c+k}\right)
\\=\underset{k,l:\mathcal{E}'\rightarrow \Z}{\sum}F_l^{\delta}(A^{c(1)+k(1)},\ldots,A^{c(n)+k(n)},A^{c(n+1)},\ldots,A^{c(n+b)})F_k^{\gamma}(A^{c(1)},\ldots,A^{c(n+b)})\varphi_{c+k+l}.
\end{multline*}
Therefore, comparing with the formula for the product in $\mathcal{A}(\Gamma)$ in Section \ref{sec:qTorus}, we get that $\sigma(\delta\cdot \gamma)=\sigma(\delta)\sigma(\gamma).$ 

The general case of $\gamma,\delta \in S(\Sigma,\Z[A^{\pm 1}])$ follows by linearity.
\end{proof}

\subsection{Injectivity of $\sigma$}
\label{sec:injectivity}
Let $\gamma$ be a multicurve, from now on, for any $k : \mathcal{E}' \to \Z$, the element $F_k^{\gamma}(Q_1,\ldots,Q_{n-b},C_1,\ldots,C_b) \in \mathcal{R}$  will be simply denoted by $F_k^{\gamma}$.

\begin{lemma} \label{lemma:fracdehn}
Let $\gamma$ be a multicurve, and let $k:\mathcal{E'}\rightarrow \Z.$ We assume that $|k_j| =\varepsilon k_j =   i(\gamma,\alpha_j) \neq 0$ with $\epsilon = \pm 1$. Let $\gamma_+$ be the curves obtained from $\gamma$ by applying a positive fractional twist along $\alpha_j$. We have $$F_k^{\gamma_+} = -A^{2\varepsilon+|k_j|} Q_j^{2\varepsilon} F_k^{\gamma}$$

\end{lemma}

\begin{proof}We will treat only the case $\varepsilon=+1,$ the case $\varepsilon=-1$ being completely similar.  Let $\gamma_-$ be the curves obtained from $\gamma$ by applying  negative fractional twist along $\alpha_j$. We have : 
$\alpha_j \gamma = A^{k_j} \gamma_+ + A^{-k_j} \gamma_- + \text{lower order curves},$ where by lower order we mean less geometric intersection with $\alpha_j.$ Hence by identifying the terms in $E^k$ in $\sigma(\alpha_j \gamma)$ we have $-(A^2 Q_j^2+A^{-2} Q_j^{-2}) E^k F_k^{\gamma} =  A^{k_j} E^k F^{\gamma_+} + A^{-k_j} E^k F_k^{\gamma_-}$ using $Q_j^2 E^k = A^{2k_j} E^k Q_j^2$ and simplify by $E^k$, we get 
$$- (A^{2+2k_j} Q_j^2+A^{-2-2k_j} Q_j^{-2})F_k^{\gamma} =  A^{k_j}  F_k^{\gamma_+} + A^{-k_j}  F_k^{\gamma_-}$$ Similarly if we expand $\gamma \alpha_j$, we get $$- (A^{2} Q_j^2+A^{-2} Q_j^{-2})F_k^{\gamma} =  A^{-k_j}  F_k^{\gamma_+} + A^{k_j}  F_k^{\gamma_-}$$
We conclude by solving the system of two equations.
\end{proof}
\begin{proposition}
$\sigma : S(\Sigma) \to \mathcal{A}(\Gamma)$ is injective. 
\end{proposition}

\begin{proof}
Let $C$ be a finite set of multicurves and $x = \underset{\gamma \in C}{\sum} \lambda_{\gamma} \gamma \in S(\Sigma)$ such that $\sigma(x) =0$. 

Let us consider an element $k$ of the set of $n$-uples of the form $(i(\gamma,\alpha_1),\ldots ,i(\gamma,\alpha_n))$ where $\gamma \in C$ which is maximal for the lexicographical order. Notice that by Proposition \ref{prop:expand}, only the multicurves $\gamma$ such that $(i(\gamma,\alpha_1),\ldots,i(\gamma,\alpha_n))=k$ contribute to the coefficient in $E^k.$ Let $C'$ be the subset of $C$ of those maximal multicurves, and let us prove that $\lambda_{\gamma}=0$ for all $\gamma\in C'.$ Note that in the Dehn-Thurston coordinates associated to the pair of pants decomposition $(\alpha_1,\ldots,\alpha_n),$ the multicurves in $C'$ differ only by their twist coordinates. If we identify the terms in $E^k$ in $\sigma(x)$, we get :
$$\sum_{\gamma \in C'} \lambda_{\gamma} F_k^{\gamma} = 0$$ Let $\delta$ be a multicurve satisfying the following conditions 
\begin{enumerate}
\item for all $1 \le j \le n$, $i(\delta,\alpha_j) = k_j$,
\item if $k_j = 0$ then the $j$-th twist coordinate of $\delta$ is trivial. 
\end{enumerate}
By Lemma \ref{lemma:fracdehn}, for each $\gamma \in C$, there exists $R_{\gamma} \in \Q(A)[Q_e^{\pm 1}]$ such that $F_k^{\gamma} = R_{\gamma} F_k^{\delta}$. It is easy to see that $\{ R_{\gamma} \, | \, \gamma \in C' \}$ are linearly independent (this follows from Lemma \ref{lemma:fracdehn} and the fact that elements in $C'$ have distinct twist coordinates). Hence from the egaility $$\sum_{\gamma \in C'} \lambda_{\gamma} R_{\gamma} F_k^{\delta} = 0,$$ we can conclude that $\lambda_{\gamma} = 0$ for all $\gamma \in C'$. An easy induction then proves that $\lambda_{\gamma}=0$ for all $\gamma \in C,$ therefore $x = 0$.
\end{proof}
We note that the embedding $\sigma$ is not surjective onto $\mathcal{A}(\Gamma).$ Indeed, as a consequence of Proposition \ref{prop:expand}-(1), the image of the morphism $\sigma$ is included in the $\Q(A)(Q_1,\ldots,Q_n)[C_1^{\pm 1},\ldots,C_b^{\pm 1}]$-subalgebra generated by elements of the form $E^k$ where $k\in \Lambda.$ In the following proposition, we show that the image of $\sigma$ is a kind of lattice in this subalgebra. Indeed the next proposition implies Proposition \ref{prop:finite_ext} in the introduction.
\begin{proposition}\label{prop:image} Let $\mathcal{F}=\Q(A)(Q_1,\ldots,Q_n)[C_1^{\pm 1},\ldots,C_b^{\pm 1}],$ then
$$\mathrm{Im}(\sigma)\underset{\mathcal{R}}{\otimes} \mathcal{F}=\underset{k\in \Lambda}{\bigoplus} \mathcal{F}E^k.$$
\end{proposition}
\begin{proof}
For $k\in \Lambda$ let $|k|=(|k_1|,\ldots,|k_n|).$ We will prove by induction on $|k|$ in the lexicographical order, that $E^k$ is a linear combination over $\mathcal{F}$ of symbols $\sigma(\gamma)$ of multicurves. The image of the empty multicurve settles the case $|k|=0.$ Next we note that since $k\in \Lambda,$ there is a multicurve $\gamma$ on $\Sigma$ such that $i(\gamma,\alpha_i)=|k_i|$ for any $1\leq i \leq n.$ For $\varepsilon\in \lbrace 0,1 \rbrace^n,$ let $\gamma_{\varepsilon}$ be the multicurve obtained from $\gamma$ by shifting its $i$-th twist coordinate by $\varepsilon_i$ if $\gamma$ has non-zero intersection with $\alpha_i.$ The curves $\gamma_{\varepsilon}$ all have $|k|$ geometric intersections with the curve $\lbrace \alpha_1,\ldots,\alpha_n\rbrace.$

For any $\mu \in \lbrace \pm 1\rbrace^n,$ and any $\varepsilon\in \lbrace 0,1 \rbrace^n,$ the coefficient $F_{\mu |k|}^{\gamma_{\varepsilon}}$ is non zero by Proposition \ref{prop:expand}-(2). Let us assume for simplicity that $\gamma$ has non-zero intersection with the curves $\alpha_1,\ldots,\alpha_d$ and is disjoint from the curves $\alpha_{d+1},\ldots,\alpha_n.$ By Lemma \ref{lemma:fracdehn}, we have 
$$F_{\mu |k|}^{\gamma_{\varepsilon}}=\underset{1\leq i \leq d}{\prod} (-A^{2\mu_i+|k_i|}Q_i^{2\mu_i})^{\varepsilon_i} F_{|k|}^{\gamma}$$ for any $\mu,\varepsilon.$
The matrix 
$$M=\left(\underset{1\leq i \leq d}{\prod} (-A^{2\mu_i+|k_i|}Q_i^{2\mu_i})^{\varepsilon_i} \right)_{\varepsilon \in \lbrace 0,1\rbrace^d, \mu \in \lbrace \pm 1\rbrace^d}$$ is the tensor product of matrices $\begin{pmatrix}
 1 & 1 \\ -A^{2+|k_i|}Q_i^2 & -A^{-2+|k_i|}Q_i^{-2}
\end{pmatrix}$ and therefore is invertible. Hence for any $\mu,$ there is an $\mathcal{F}$-linear combination of the multicurves $\gamma_{\varepsilon}$ such that its image by $\sigma$ has coefficient $1$ along $E^{\mu |k|}$ and zero coefficient along each other $E^{\mu'|k|}$ where $\mu'\neq \mu \in \lbrace \pm 1 \rbrace^d.$ As a result, we get an $\mathcal{F}$-linear combination of multicurves $x=\sum \lambda_{\varepsilon} \gamma_{\varepsilon}$ so that $\sigma(x)=E^{\mu |k|}$ up to lower order terms, and by induction hypothesis we can add another linear combination of multicurves to eliminate those lower order terms. 
\end{proof}
\begin{remark}\label{rk:finExt} In some sense, the embedding is analogous to the Frohman-Gelca embedding \cite{FG00} of the skein algebra of the closed torus into the quantum torus $\mathbb{Z}[A^{\pm 1}]\langle Q,E\rangle/_{QE=AEQ}.$ The image of the Frohman-Gelca embedding is the symmetric part of the quantum torus; that is elements invariant under the action of the $\mathbb{Z}[A^{\pm 1}]$-algebra automorphism $\theta : Q,E\longmapsto Q^{-1},E^{-1}.$ Here, since $\sigma(\alpha_e)=-A^2Q_e^2-A^{-2}Q_e^{-2},$ , we have that the localized quantum torus is a kind of "finite extension" of the skein algebra by Proposition \ref{prop:image}.
\end{remark}
\subsection{Proof of Proposition \ref{prop:expand}}
\label{sec:curveOpCoeff}
\begin{figure}[!h]
\centering
\def \svgwidth{.7\columnwidth}
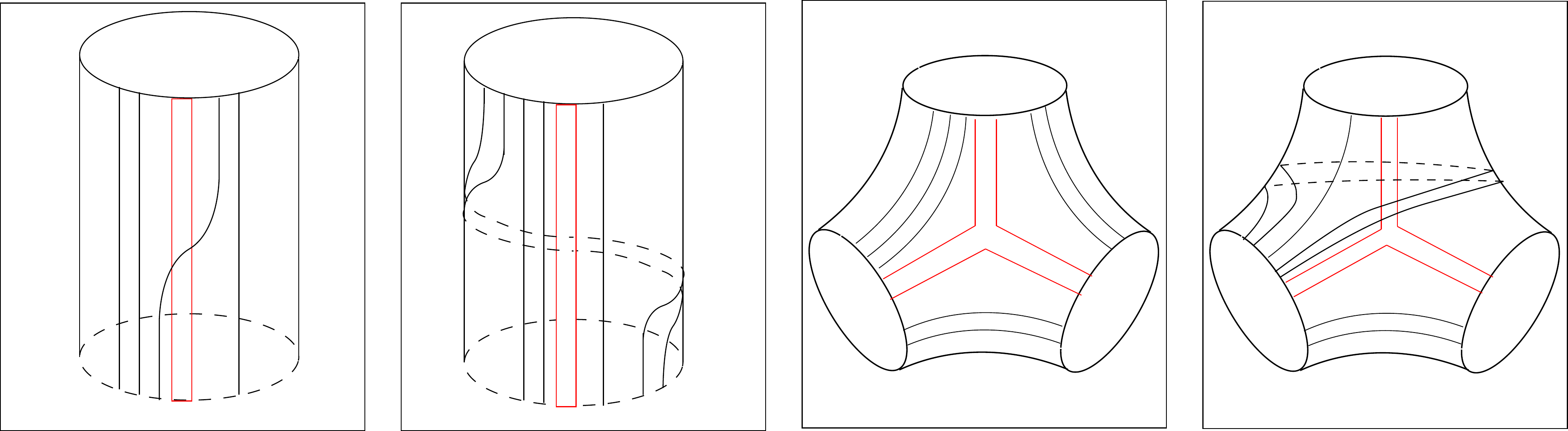
\caption{A multicurve in Dehn-Thurston position follows one of the above patterns in the pants and annuli of the decomposition.}
\label{fig:DehnThurston}
\end{figure}
\begin{figure}[h]
\centering
\def \svgwidth{.75\columnwidth}
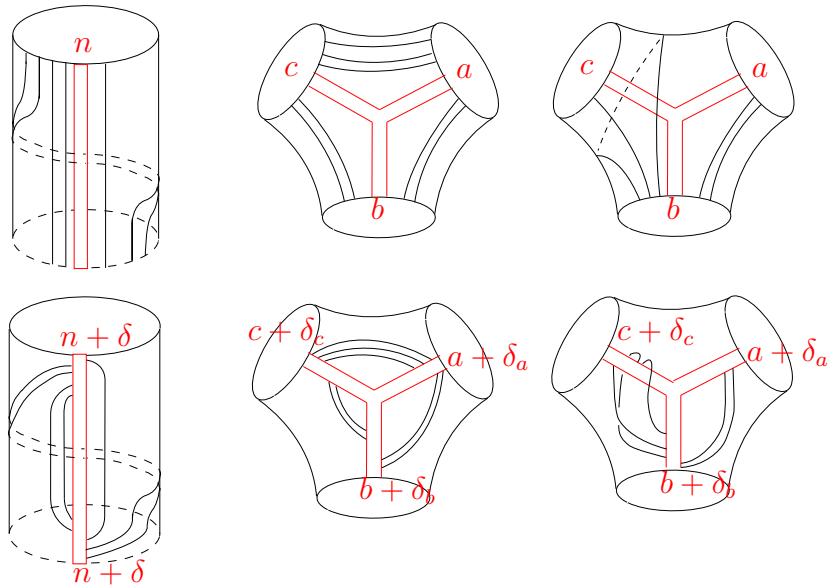
\caption{On the top the different patterns of the intersection of a multicurve (in black) with an annulus or pants piece of the decomposition. The trivalent graph $\Gamma$ is shown in red. On the bottom the remaining patterns after fusion.}
\label{fig:Patterns}
\end{figure}

In this section, we will prove Proposition \ref{prop:expand}. Let $\Sigma$ be a compact oriented surface of genus $g$ and with $b$ boundary components with negative Euler characteristic, $\mathcal{P}=\lbrace \alpha_1,\ldots,\alpha_{3g-3+b}\rbrace$ be a pair of pants decomposition of $\Sigma$ and $\Gamma$ be a trivalent banded graph dual to $\mathcal{P}.$ We also view $\Sigma$ as the boundary of a handlebody $H,$ thus vectors $\varphi_c$ associated to colorings of $\Gamma$ give a basis of $S(H,\Q(A))$ by Lemma \ref{lemma:basis}.

 Taking two parallel copies of each curve in $\mathcal{P},$ we get a decomposition of $\Sigma$ into pair of pants and annuli. We will call $\alpha_i$ and $\alpha_i'$ the two parallel copies, with no particular convention for the choice of $\alpha_i.$ Note that the pair of pants in the decomposition have boundary either the curves $\alpha_i$,$\alpha_i'$ or the boundary curves of $\Sigma.$
 
 If $\gamma$ is a multicurve on $\Sigma,$ then up to isotopy $\gamma$ can be put into Dehn-Thurston position. By this we mean that the geometric intersection number $i(\gamma,\alpha_i)$ of $\gamma$ with each curve in $\mathcal{P}$ is exactly the number of intersection points of $\gamma$ with $\alpha_i$ and also with $\alpha_i',$ and that in each pair of pants or annuli of the decomposition the curve $\gamma$ looks like one of the patterns described in Figure \ref{fig:DehnThurston}.
 
The computation of the coefficients of the operator $T^{\gamma}$ in the basis $\varphi_c$ can be done as follows. First we can remove any component $\beta$ in $\gamma$ that is parallel to a curve in $\mathcal{P}$ or a boundary curve, at the price of multiplying $\varphi_c$ by the scalar $-A^{2c_e+2}-A^{-2c_e-2}=-A^2Q_e^2-A^{-2}Q_e^{-2}$ where $e$ is the edge encircled by $\beta.$ 

Secondly we apply the first fusion rule for each intersection point of $\gamma$ with a curve $\alpha_i$ or $\alpha_i'.$ Each fusion shifts the color of the corresponding edge by $\pm 1.$ After fusion, it only remains to simplify one of the patterns described in Figure \ref{fig:Patterns} to express the coefficients of $T^{\gamma}\varphi_c$ on the basis $\lbrace \varphi_c \rbrace.$ Note that as the skein module of the sphere with two colored points is one-dimensional if the two colors agree and zero else, to obtain a non-zero vector we need that the sums of shifts at intersection points with $\alpha_i$ and $\alpha_i'$ coincide. Furthermore, the skein module of a sphere with three points colored by $c_1,c_2,c_3$ is one-dimensional if the colors satisfy the admissibility conditions $c_i\leq c_j+c_k$ and $c_1+c_2+c_3$ even, and zero-dimensional else. By the assumption that $c$ is in the set $V_{\gamma},$ that is always the case after fusion. Therefore, each of the terms we obtain after fusion at the intersection points of $\gamma$ with $\alpha_i$ and $\alpha_i'$ is just a scalar multiple of a vector $\varphi_{c+k}.$ The color shift $k_e$ at edge $e$ is the common sum of the $\pm 1$ shifts at either intersection points in $\gamma\cap \alpha_i$ or at intersection points in $\gamma\cap \alpha_i'$. This shows that $T^{\gamma}\varphi_c$ has non-zero coefficient along $\varphi_{c+k}$ only when $k \in \Lambda$ and $|k_e|\leq i(\gamma,c_e)$ for each edge $e.$

Next we claim that the coefficients are in the ring $\mathcal{R}.$ Notice that the remaining patterns after fusion at intersection points $\gamma\cap(\alpha_e\cup \alpha_e')$ shown in Figure \ref{fig:Patterns} can be reduced to remove all black arcs using the fusion rules in Figure \ref{fig:fusion}. Furthermore, all of the fusion rules in Figure \ref{fig:fusion} involve only rationals functions of $A^{c_e}$ and $A,$ where $c_e$ are the colors of edges $e\in \mathcal{E}.$ Moreover, the denominators appearing in the fusion rules are all of the form $\lbrace c_e+k \rbrace=A^{2c_e+2k}-A^{-2c_e-2k}$ where $k\in \Z.$ (We remark that the first fusion rules will shift colors of $\Gamma$ but only by a fixed amount). We also claim that if $e$ is an external edge, we will never need to use any rule involving a denominator $\lbrace c_e+k\rbrace,$ since the geometric intersection of $\gamma$ and $\alpha_e$ is zero. Those rules correspond the rules that create or erase a black arc with an endpoint on the edge $e.$ Therefore, the coefficients are in the ring $\mathcal{R}.$

It remains to be seen that the extremal coefficients are non-zero. This is a consequence of Theorem 1.3 of \cite{Det16} and Lemma 4.3 of \cite{CM12}. This first theorem studies the matrix coefficients of the action of curve operators on $\mathrm{SU}_2$ TQFT spaces of surfaces. A subset of the basis $\varphi_c$ of the skein module of the handlebody, corresponding to colors $c$ satisfying the additional "$r$-admissibility conditions" gives a basis of the $\mathrm{SU}_2$ TQFT space of the surface $(\Sigma,c_i)$ with colored points. The curve operators on $\mathrm{S}$ induce curve operators $T_r^{\gamma}$ on the TQFT spaces at level $r$, and their matrix coefficients $F_k^{\gamma,\mathrm{SU}_2}(\frac{c}{r},\frac{1}{r})$ are obtained from the coefficients of $T^{\gamma}$ in the $E^k$ by sending $A$ to a $2r$-th root of unity and applying some renormalization. Theorem 1.3 of \cite{Det16} then shows that $F_k^{\gamma,\mathrm{SU}_2}(x,0)$ is the $k$-th Fourier coefficient of the trace function $f_{\gamma} : \rho \mapsto \underset{i\in I}{\prod} (-\mathrm{Tr}(\rho(\gamma_i))$ defined on the subset $\lbrace \forall e\in \mathcal{E}', \mathrm{Tr}(\rho(\alpha_e))=2\mathrm{cos}(\pi x_e)) \rbrace$ of the $\mathrm{SU}(2)$ moduli space of $\Sigma.$ Lemma 4.3 of \cite{CM12} shows that the extremal Fourier coefficients are non-zero as long as $x$ is taken in the interior of the image of the momentum map $\rho \mapsto \left(\mathrm{arcos}(\frac{\mathrm{Tr}(\rho(\alpha_e))}{2})\right)_{e\in \mathcal{E}'}.$ This implies that the coefficients $F_k^{\gamma,\mathrm{SU}_2}$ are non-zero when $k_e=\pm i(\gamma,\alpha_e)$ for all $e\in \mathcal{E}',$ and the same is true for the coefficients $F_k^{\gamma}$.
\section{Localized quantum torus associated to a sausage graph}
\label{sec:sausage_case}
In this section, as well as the remaining of the paper, we introduce the notation $U(x)=x-x^{-1}$ when $x$ is an invertible element in a ring.
Let $\Sigma$ be a surface with at most one boundary component with negative Euler characteristic. Consider the pants decomposition of $\Sigma$ as in Figure \ref{fig_sausage}. 
\begin{figure}[h!] 
\centering
\includegraphics[scale = 0.35]{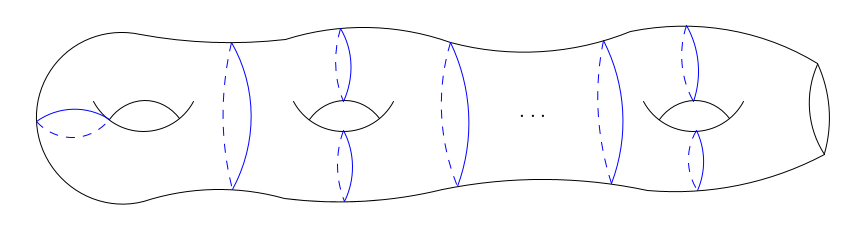}
\caption{When $\partial \Sigma = \emptyset$, the two rightmost curves coincide.}
\label{fig_sausage}
\end{figure}
Let $\Gamma$ be the graph dual to this pants decomposition :
$$\includegraphics[scale = 0.45]{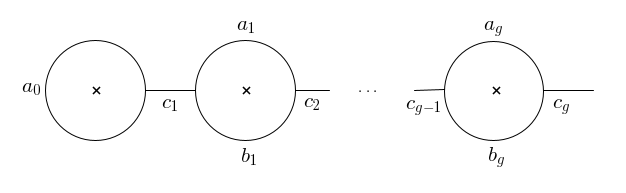}$$
When $\partial \Sigma = \emptyset$, $a_{g-1} = b_{g-1}$ and $c_g$ does not exist. Finally let 
$\beta_1,\ldots,\beta_g,\gamma_1,\ldots,\gamma_{g-1}$ be the curves shown in Figure \ref{fig_beta_gamma}.
\begin{figure}[h!] 
\centering
\includegraphics[scale = 0.45]{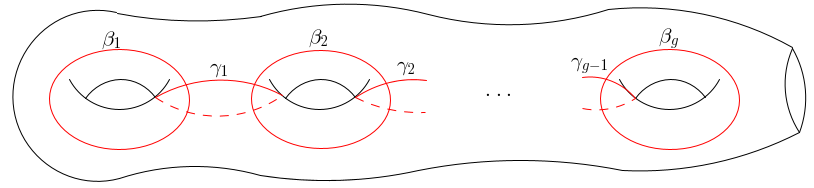}
\caption{The $\beta$ and $\gamma$ curves}
\label{fig_beta_gamma}
\end{figure}

For $k : \mathcal{E}' \to \Z$, we define $E^k$ to be $\prod_{e \in \mathcal{E}'} E_e^{k(e)}$ and $\Lambda$ be the set of maps $k : \mathcal{E}' \to \Z$ such that if $e_1,e_2,e_3 \in E(\Gamma)$ meet at a vertex then $k(e_1)+k(e_2)+k(e_3)$ is even. 

\begin{definition} \label{A0}
Let $\mathcal{R}^0$ be the set of Laurent polynomial with coefficients in $Z[A^{ \pm 1}]$ in the variables $Q_{a_0}^2,Q_{a_1} Q_{b_1},  Q_{a_1} Q_{b_1}^{-1}, \ldots, Q_{a_{g-1}} Q_{b_{g-1}},  Q_{a_{g-1}} Q_{b_{g-1}}^{-1}, Q_{c_1},\ldots,Q_{c_g}$. We define $\mathcal{A}(\Gamma)^0$ be the sub-algebra of $\mathcal{A}(\Gamma)$ defined by the set of 
$$\sum_{k \in \Lambda} E^k F_k$$ where $F_k = V/W$ with a $V \in \mathcal{R}^0$ and $W$ is a finite (possibly empty) product of $A^n Q_{\alpha}^2-A^{-n} Q_{\alpha}^{-2}$ for $n \in \Z$ and $\alpha \in \mathcal{P}$.
\end{definition}
We define also the \textit{non localized} version $\mathcal{A}(\Gamma)^0$ :
$$\mathcal{T}(\Gamma)^0 = \left\{ \sum_{k \in \Lambda} E^k F_k \, ; \, F_k \in \mathcal{R}^0 \right\} \subset \mathcal{A}(\Gamma)^0$$

\begin{lemma} \label{gen_t}
$\mathcal{T}(\Gamma)^0$ is generated by the following sets \begin{align*} \mathcal{X}= & \{Q_{a_0}^2,Q_{a_1} Q_{b_1},  Q_{a_1} Q_{b_1}^{-1}, \ldots, Q_{a_{g-1}} Q_{b_{g-1}},  Q_{a_{g-1}} Q_{b_{g-1}}^{-1}, Q_{c_1},\ldots,Q_{c_{g-1}} \} \\ 
\mathcal{Y} = & \{E_{a_0},E_{a_1} E_{b_1},  E_{a_1} E_{b_1}^{-1}, \ldots, E_{a_{g-1}} E_{b_{g-1}},  E_{a_{g-1}} E_{b_{g-1}}^{-1}, E_{c_1}^2,\ldots,E_{c_{g-1}}^2 \} \\
\mathcal{Z} = &  \{Q_{c_g} \}  \end{align*} and the inverses of the elements of these sets.
\end{lemma}

\noindent \textbf{Convention :} From now on, we will use the same symbol for an edge of $\Gamma$ and the unique curve in $\mathcal{P}$ encircling it.
\\

Recall that $\sigma : S(\Sigma) \to \mathcal{A}(\Gamma)$ is an embedding. We see that for $e \in \mathcal{P}$ $$\sigma(e) = -(A^2 Q_e^2 + A^{-2} Q_e^{-2})$$ Let us give an explicit expression of $\sigma(\gamma)$ for $\gamma \in \{\beta_1,\ldots,\beta_g,\gamma_1,\ldots,\gamma_{g-1}\}$. There are three local configurations for $\gamma$ as shown in Figure \ref{fig_cycle}.
\begin{center}
\begin{figure} 
\centering
\subfigure[One-cycle \qquad]{\includegraphics[width=3cm]{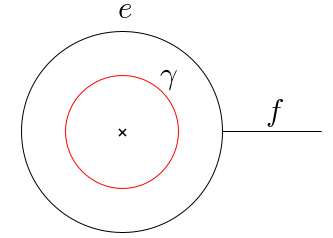}} 
\qquad
\subfigure[Two cycle]{\includegraphics[width=3.5cm]{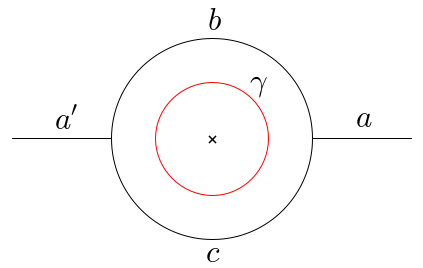}} 
\qquad
\subfigure[Separating edge curve]{\includegraphics[width=4cm]{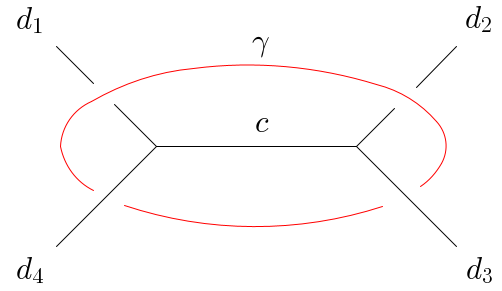}} 
\qquad
\caption{Three configurations for the curve $\gamma$}
\label{fig_cycle}
\end{figure}
\end{center}
Suppose that $\gamma$ is a one-cycle as in Figure \ref{fig_cycle} (a). The fusion rules say that : 
\begin{equation} \label{emb_1_cycle} \sigma(\gamma) = E_e+E_e^{-1}F  \end{equation} where 
$$F = U(A^2 Q_e^2 Q_f)U(Q_e^2 Q_f^{-1})U(A^2 Q_e^2)^{-1}U(Q_e^2)^{-1}$$
and we recall the notation $U(x)=x-x^{-1}.$

Suppose that $\gamma$ is a two-cycle as in Figure \ref{fig_cycle} (b). The fusion rules say that : 
\begin{equation} \label{emb_2_cycle} \sigma(\gamma) = E_b E_b +E_b E_c^{-1} F_{1,-1}+E_b^{-1} E_c F_{-1,1} + E_b^{-1} E_c^{-1} F_{-1,-1} 
\end{equation}

where 
\begin{align*}
F_{-1,1} & = -\dfrac{U(Q_{a'} Q_{b} Q_{c}^{-1}) U(Q_{a} Q_{b} Q_{c}^{-1})}{U(A^2 Q_{b}^2)U(Q_{b}^2)}  \\
F_{1,-1} & = -\dfrac{U(Q_{a'} Q_{c} Q_{b}^{-1}) U(Q_{a} Q_{c} Q_{b}^{-1})}{U(A^2 Q_{c}^2)U(Q_{c}^2)}  \\
F_{-1,-1} & = \dfrac{U(A^2 Q_{a'} Q_{c} Q_{b}) U(A^2 Q_{a} Q_{c} Q_{b}) U(Q_{b} Q_{c} Q_{a'}^{-1})U(Q_{b} Q_{c} Q_{a}^{-1})}{U(A^2 Q_{c}^2)U(Q_{c}^2)U(A^2 Q_{b}^2)U(Q_{b}^2)} 
\end{align*}

Finally suppose that $\gamma$ is a separating edge curve as in Figure \ref{fig_cycle} (c), we have : 
\begin{equation} \label{emb_sep} \sigma(\gamma) = E_c^2 G_2+ G_0+ E_c^{-2} G_{-2}
\end{equation} where 
\begin{align*}
G_0  = & \frac{(d_1 d_3 +d_2 d_4)c +(A^2+A^{-2})(d_1 d_2 + d_3 d_4)}{U(Q_c^2) U(A^4 Q_c^2)}\\
& \\
G_{-2}  =&  -\dfrac{U(A^2 Q_{d_1}Q_{d_4} Q_c)U(Q_{d_1}Q_c Q_{d_4}^{-1})U(Q_{d_4}Q_c Q_{d_1}^{-1})}{U(A^{-2}Q_c^2) U(Q_c^2)^2 U(A^2 Q_c^2)} \\
 & \times U(A^2 Q_{d_2}Q_{d_3} Q_c)U(Q_{d_2}Q_c Q_{d_3}^{-1})U(Q_{d_3}Q_c Q_{d_2}^{-1})\\
& \\
G_2  = & -U(Q_{d_1} Q_{d_4} Q_{c}^{-1}) U(Q_{d_2} Q_{d_3} Q_{c}^{-1}) 
\end{align*}
with $d_j = -A^2 Q_{d_j}^2-A^{-2} Q_{d_j}^{-2}$ for $j \in \{1,2,3,4\}$.
\begin{proposition}
If $\xi$ is a non zero complex number such that $\xi^4 \neq 1$ then the image $\sigma_{\xi} : S_{\xi}(\Sigma) \to \mathcal{A}_{\xi}(\Gamma)$ lies in $\mathcal{A}_{\xi}(\Gamma)^0$.
\end{proposition}
\begin{proof}
Let $\alpha \in \mathcal{P}$, $\sigma_{\xi}(\alpha) = -(\xi^2 Q_e^2+\xi^{-2} Q_e^{-2})$ for some edge $e \in E(\Gamma)$, therefore $\sigma_{\xi}(\alpha) \in  \mathcal{A}_{\xi}(\Gamma)^0$. Also $\sigma_{\xi}(\gamma) \in  \mathcal{A}_{\xi}(\Gamma)^0$ for any $\gamma \in \{\beta_1,\ldots,\beta_g,\gamma_1,\ldots,\gamma_{g-1}\}$ according to Formulas \ref{emb_1_cycle},\ref{emb_2_cycle} and \ref{emb_sep}. Since the set of Dehn twists associated to the curves in $\mathcal{P}\cup \{\beta_1,\ldots,\beta_g,\gamma_1,\ldots,\gamma_{g-1}\}$ generate the mapping class group of $\Sigma$, the set $\mathcal{P}\cup \{\beta_1,\ldots,\beta_g,\gamma_1,\ldots,\gamma_{g-1}\}$ generate $S_{\xi}(\Sigma)$ by \cite[Thm 1.1]{San}. We can then conclude.
\end{proof}

\begin{figure}[t] 
\centering
\includegraphics[scale = 0.35]{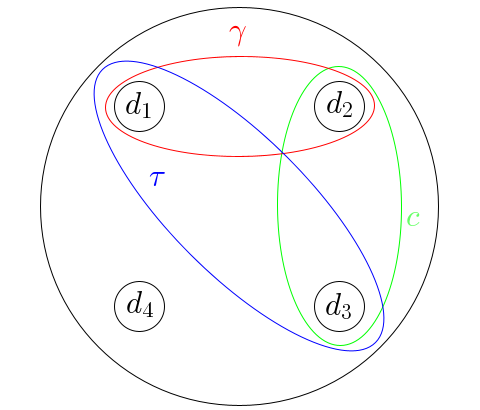}
\caption{The curve $\tau$}
\label{fig_tau}
\end{figure}
The next proposition is an explicit version of Proposition \ref{prop:finite_ext}.
\begin{proposition}
Let $\gamma \in \{\beta_1,\ldots,\beta_g,\gamma_1,\ldots,\gamma_{g-1}\}$. If $\gamma$ is a one-cycle as in Figure \ref{fig_cycle} (a), then 
\begin{equation} \label{1_cycle_lift} \left\{
    \begin{array}{llll}
E_e & = - \Big(t_{e}(\gamma)+A^{-1}\gamma Q_e^{-2}\Big)A^{-1} U(A^2 Q_e^2)^{-1}  \\
E_e^{-1} & =\quad \Big(A^3 \gamma Q_e^2+t_{e}(\gamma) \Big)A^{-1} U(A^2 Q_e^2)^{-1}F^{-1}  \end{array}
\right.
\end{equation} where $F$ was defined in Equation (\ref{emb_1_cycle}). If $\gamma$ is a two-cycle as in Figure \ref{fig_cycle} (b), then 

\begin{align}
&E_b E_c  = \Big( \sigma(\gamma) A^{-2} Q_b^{-2} Q_c^{-2}+ \sigma(t_b(\gamma)) A^{-1} Q_c^{-2} + \sigma(t_c(\gamma)) A^{-1} Q_b^{-2} +\sigma(t_b t_c(\gamma)) \Big) D^{-1}  \label{2_cycle_lift1}
\\
&E_b E_c^{-1} F_{1,-1} = -\Big( \sigma(\gamma) A^{2} Q_b^{-2} Q_c^{2}+ \sigma(t_b(\gamma)) A^{3} Q_c^{2} + \sigma(t_c(\gamma)) A^{-1} Q_b^{-2} +\sigma(t_b t_c(\gamma)) \Big) D^{-1} \label{2_cycle_lift2}
\\
&E_b^{-1} E_c F_{-1,1} =  -\Big( \sigma(\gamma) A^{2} Q_b^{2} Q_c^{-2}+ \sigma(t_b(\gamma)) A^{-1} Q_c^{-2} + \sigma(t_c(\gamma)) A^{3} Q_b^{2} +\sigma(t_b t_c(\gamma)) \Big) D^{-1} \label{2_cycle_lift3}
\\
&E_b^{-1} E_c^{-1}F_{-1,-1} = \Big( \sigma(\gamma) A^{6} Q_b^{2} Q_c^{2}+ \sigma(t_b(\gamma)) A^{3} Q_c^{2} + \sigma(t_c(\gamma)) A^{3} Q_b^{2} +\sigma(t_b t_c(\gamma)) \Big) D^{-1} \label{2_cycle_lift4}
\end{align}
where $D = A^2 U(A^2 Q_c^2) U(A^2 Q_b^2)$ and where $F_{-1,1},F_{1,-1},F_{-1,-1}$ were defined in Equation (\ref{emb_2_cycle}). Finally if $\gamma$ is a separating edge curve as in Figure \ref{fig_cycle} (c), then
\begin{align}
& E_c^2 G_2  = -\left[\gamma Q_c^{-2} + \tau - \dfrac{A^{-2} \delta_1 Q_c^{-2}-A^2 \delta_2}{U(A^4 Q_c^2)}\right] A^{-2} U(A^2 Q_c^2)^{-1}        \label{sep_lift1} \\
& E_c^{-2} G_{-2}  = \left[\gamma A^2 Q_c^2 + A^{-2} \tau + \dfrac{ \delta_1 Q_c^2-\delta_2}{U(Q_c^2)}\right]  U(A^2 Q_c^2)^{-1}  \label{sep_lift2}
\end{align}
where $\delta_1 = d_1 d_3+d_2 d_4$, $\delta_2 = d_1 d_2+d_3 d_4$ and $\tau$ is as in Figure \ref{fig_tau}.
\end{proposition}

\begin{proof}
Suppose that $\gamma$ is a one-cycle. Recall that $\sigma(e) = -(A^2 Q_e^2+A^{-2} Q_e^{-2})$ and $ Q_e E_e = A E_e Q_e$. Moreover in the skein algebra of $\Sigma$, we have $A e \gamma -A^{-1} \gamma e = (A^2-A^{-2})t_e(\gamma)$. Using Equation (\ref{emb_1_cycle}) we obtain the following $\sigma(t_{e}(\gamma)) =  -A^3 E_e Q_e^2 - A^{-1} E_e^{-1} Q_e^{-2} F$ and therefore the following system
$$\left\{
    \begin{array}{llll}
    \sigma(\gamma) & = E_e  +  E_e^{-1}  F \\
\sigma(t_{e}(\gamma)) & = -A^3 E_e Q_e^2 - A^{-1} E_e^{-1} Q_e^{-2} F 
 \end{array}
\right.
$$ Solving this system gives Equations (\ref{1_cycle_lift}).

In the case where $\gamma$ is a two-cycle, the strategy is very similar. In this case, we apply $t_b$, $t_c$ and $t_b t_c$ to the curve $\gamma$ to get a system of four equations whose resolution gives (\ref{2_cycle_lift1}),
(\ref{2_cycle_lift2}),(\ref{2_cycle_lift3}), (\ref{2_cycle_lift4}).

Suppose now that  $\gamma$, is separating edge curve. In the Skein algebra of $\Sigma$ we have $$A^2 c \gamma - A^{-2} \gamma c = (A^4-A^{-4}) \tau + (A^2-A^{-2}) (d_1 d_3 +d_2 d_4)$$ Applying this to Equation (\ref{emb_sep}), one gets the following system
$$\left\{
    \begin{array}{llll}
\sigma(\gamma) -G_0 & = E_c^2 G_2+ E_c^{-2} G_{-2} \\
\sigma(\tau) + \dfrac{\sigma(d_1 d_3+d_2 d_4)-G_0 \sigma(c)}{A^2+A^{-2}} & = -E_c^{2}G_2 A^4 Q_c^2 -  E_c^{-2}G_{-2}  Q_c^{-2}  \end{array}
\right.
$$
Solving this system gives Equation (\ref{sep_lift1}) and (\ref{sep_lift2}).
\end{proof}
We finish this section by analyzing irreducible representations of $\mathcal{A}(\Gamma)^0$ at root of unity. Let $\xi$ be a $2p$-th primitive root of unity with $p$ odd and let $\rho : \mathcal{A}_{\xi}(\Gamma)^0 \to \mathrm{End}(V)$ be a complex irreducible finite dimensional representation. Let $x \in \mathcal{X} \cup \mathcal{Y}$, notice that $x^p$ is a central element of $\mathcal{A}_{\xi}(\Gamma)^0$, hence $\rho(x^p)$ is a scalar times the identity of $V$ (because $\rho$ is irreducible). Let us denote by $\lambda_{x,\rho}$ this scalar. Note that $Q_{c_g}$ is also central so $\rho(Q_{c_g}) = \lambda_{\rho} \mathrm{Id}_V$ for some $\lambda_{\rho} \in \C$.
\begin{proposition} \label{rep_Q_t}
Let $\rho_1$ and $\rho_2$ be two complex irreducible finite dimensional representations of $\mathcal{A}_{\xi}(\Gamma)^0$. If $\lambda_{\rho_1} = \lambda_{\rho_2}$ and $\lambda_{x,\rho_1}=\lambda_{x,\rho_2}$ for all $x \in \mathcal{X} \cup \mathcal{Y}$ then $\rho_1$ and $\rho_2$ are isomorphic. Any irreducible representation of $\mathcal{A}_{\xi}(\Gamma)^0$ has dimension $p^{3g-2}$ when $\Gamma$ has an univalent vertex and dimension $p^{3g-3}$ otherwise. 

\end{proposition}

\begin{proof}
It is enough to prove that the restrictions of $\rho_1$ and $\rho_2$ to $\mathcal{T}_{\xi}(\Gamma)^0$ are isomorphic. Let $\mathcal{W}$ be the complex algebra defined by the generators $U^{\pm 1}$, $V^{\pm 1}$ and by the relation $UV = \xi^2 V U$. From the generators of Lemma \ref{gen_t} it is easy to see that $\mathcal{T}_{\xi}(\Gamma)^0$ is isomorphic to $\mathcal{W}^{\otimes 3g-3}$ when $\Gamma$ does not have a univalent vertex (which is when $\partial \Sigma = \emptyset$) and to $\mathcal{W}^{\otimes 3g-2} \otimes \C[Z^{ \pm 1}]$ (where $Z$ is a formal independent variable) when $\Gamma$ has one univalent vertex. The result follows directly from \cite[Lemma 17,18]{BL07}.
\end{proof}

\section{Representation of the skein algebra}
\label{sec:representations}
In this section $\Sigma$ is still a surface with at most one boundary component with negative Euler characteristic. The goal of this section is to prove Theorem \ref{thm:lift_rep} and Corollary \ref{cor:unicity}.

Let $p \ge 3$ an odd number, we warn the reader that from now on $A$ will not be a formal variable but a $2p$-th primitive root of unity. We recall the Bonahon-Wong theory,  let $$\rho : S_A(\Sigma) \to \mathrm{End}(V)$$ be a finite dimensional irreducible representation. By Bonahon-Wong work, there exists $r : \pi_1(\Sigma) \to \mathrm{SL}_2(\C)$ such that for any simple close curve $\gamma$ we have \begin{equation} \label{1} T_p(\rho(\gamma)) = -\mathrm{Tr}(r(\gamma))\mathrm{Id}_V \end{equation} Here $T_k$ is $k$-th Chebyshev polynomial of the first kind, the important thing to remember is that $$T_k(u+u^{-1}) = u^k + u^{-k} \quad \forall u\in \C-\{0 \}$$ Moreover $r$ is called the classical shadow of $\rho$.

The proof of Theorem \ref{thm:lift_rep} requires several steps, it will starts at Subsection \ref{subs1} and ends at Subsection \ref{proof_lift_rep}. The goal is to define $\rt$ on the generators of $\mathcal{A}_A(\Gamma)^0$ (given in Lemma \ref{gen_t}), prove the relations satisfied by these generators are preserved by $\rt$, and show that $\rt$ agrees with $\rho$ on $\sigma_A(\Sigma_A(\Sigma))$.

The proof of Corollary \ref{cor:unicity} is done in Subsection \ref{classical_shadow}. From now on, we fix a irreducible representation $\rho : S_A(\Sigma) \to \mathrm{End}(V)$ with classical shadow $r$ satisfying the hypothesis of Theorem \ref{thm:lift_rep}.

\subsection{Action of the $Q$ operators} \label{subs1}
For $\alpha \in \mathcal{P}$, let us chose $x_{\alpha} \neq 0$ such that $\mathrm{Tr}(r(\alpha)) = x_{\alpha}^{2p} + x_{\alpha}^{-2p}$. A known fact is that $(1)$ implies that $\rho(\alpha)$ is diagonalizable with eigenvalues $$-(x_{\alpha}^2 A^{2k+2} + x_{\alpha}^{-2} A^{-2k-2})$$ for $k = 0,\ldots,p-1$. We define for $k \in \Z$ $$V_{\alpha,k} = \mathrm{Ker}\Big(\rho(\alpha)+(x_{\alpha}^2 A^{2k+2} + x_{\alpha}^{-2} A^{-2k-2} ) \mathrm{Id}_V \Big)$$ 
We also define $\tilde{\rho}(Q_{\alpha}) \in \mathrm{GL}(V)$ by $$\rt(Q_{\alpha}) v = x_{\alpha} (-A)^k v \quad \quad \forall v \in V_{\alpha,k}$$ so that 
\begin{equation} \label{Q}
\rho(\alpha) = -(A^2 \rt(Q_{\alpha})^2+ A^{-2} \rt(Q_{\alpha})^{-2})
\end{equation}
As the matrices $\{\rho(\alpha) \, | \, \alpha \in \mathcal{P} \}$ pairwise commute, it is clear that $\{ \rt(Q_{\alpha}) \, | \, \alpha \in \mathcal{P} \}$ pairwise commute. Hence the set $\{ \rt(Q_{\alpha}) \, | \, \alpha \in \mathcal{P} \}$ defines a morphism $\rt : \C[Q_e^{\pm 1} \, | \, e \in \mathcal{E}(\Gamma)] \to \mathrm{End}(V)$ (here we use that $\mathcal{P}$ is canonically in bijection with $\mathcal{E}(\Gamma)$).

The following lemma will help us extend $\rt$ further. Recall that for an invertible element $W$, $U(W)= W-W^{-1}$.

\begin{lemma}
For all $\alpha \in \mathcal{P}$, we have $U(A^k \rt(Q_{\alpha})^2) \in \mathrm{GL}(V)$ for any $k \in \Z$.
\end{lemma}
\begin{proof}
Let $k \in \Z$ and $\alpha \in \mathcal{P}$, the eigenvalues of $U(A^k \rt(Q_{\alpha})^2)$ are $U(A^{k +2l} x_{\alpha}^2)$ for $l \in \Z$. Suppose that one of these is $0$ which is to say that $U(A^k \rt(Q_{\alpha})^2) \notin \mathrm{GL}(V)$. It would imply $A^{k +2l} x_{\alpha}^2=A^{-k -2l} x_{\alpha}^{-2}$ for some $l \in \Z$. Taking the $p$-th power of this equality gives $x_{\alpha}^{4p} =1$ which implies $r(\alpha) = x_{\alpha}^{2p}+x_{\alpha}^{-2p}= \pm 2$. This would contradict (\ref{eqG2}).
\end{proof}
From this lemma we see that $\rt$ is defined for elements in $\mathcal{A}_A(\Gamma)$ of the form $\dfrac{X}{Y}$ where $X,Y \in \C[Q_e^{\pm 1} \, | \, e \in \mathcal{E}(\Gamma)]$ and $Y$ is a finite product of elements in the set $\{U(A^k Q_{\gamma}) \, | \, \gamma \in \mathcal{P}, k \in \Z \}$.

\begin{lemma} \label{invF}
Let $e_1,e_2,e_3 \in E(\Gamma)$ meeting at vertex. For all $k \in \Z$ and $\epsilon_1,\epsilon_2,\epsilon_3 \in \{-1,1 \}$ we have $$\rt(U(A^k Q_{e_1}^{\epsilon_1}Q_{e_2}^{\epsilon_2} Q_{e_3}^{\epsilon_3})) \in \mathrm{GL}(V)$$
\end{lemma}
\begin{proof}
Let $\alpha_1,\alpha_2,\alpha_3 \in \mathcal{P}$ respectively dual to $e_1,e_2,e_3$.
Suppose there exists $k \in \Z$ and $\epsilon_1,\epsilon_2,\epsilon_3 \in \{-1,1 \}$ such that $\rt(U(A^k Q_{e_1}^{\epsilon_1}Q_{e_2}^{\epsilon_2} Q_{e_3}^{\epsilon_3}))$ is not invertible. It would mean that for some $l \in \Z$ we have $U(\pm A^l x_{\alpha_1}^{\epsilon_1} x_{\alpha_2}^{\epsilon_2} x_{\alpha_3}^{\epsilon_3}) = 0$. This implies $A^l x_{\alpha_1}^{\epsilon_1} x_{\alpha_2}^{\epsilon_2} x_{\alpha_3}^{\epsilon_3}=A^{-l} x_{\alpha_1}^{-\epsilon_1} x_{\alpha_2}^{-\epsilon_2} x_{\alpha_3}^{-\epsilon_3}$, taking the $p$-th power this equality (and remembering that $A^p = A^{-p}$), one gets $$ x_{\alpha_1}^{p\epsilon_1} x_{\alpha_2}^{p\epsilon_2} x_{\alpha_3}^{p\epsilon_3}= x_{\alpha_1}^{-p\epsilon_1} x_{\alpha_2}^{-p\epsilon_2} x_{\alpha_3}^{-p \epsilon_3}$$ Let $\lambda_1 = x_{\alpha_1}^{p\epsilon_1}, \lambda_2 = x_{\alpha_2}^{p\epsilon_2}, \lambda_3 =  x_{\alpha_3}^{p\epsilon_3}$ and remember that $\lambda_1 \lambda_2 \lambda_3 - \lambda_1^{-1} \lambda_2^{-1} \lambda_3^{-1} = 0$. Now notice that 
$$\prod_{\delta_2,\delta_3 \in \{-1,1\} } (\lambda_1 \lambda_2^{\delta_2} \lambda_3^{\delta_3}- \lambda_1^{-1} \lambda_2^{-\delta_2} \lambda_3^{-\delta_3}) = 2+\sum_{k=1}^3(\lambda_k^4+\lambda_k^{-4})- \prod_{k=1}^3 (\lambda_k^2+\lambda_k^{-2})$$
with the left hand side being zero by assumption and the right hand side being 
$$2+\sum_{k=1}^3 \mathrm{Tr}(r(\alpha_k^2))- \prod_{k=1}^3 \mathrm{Tr}(r(\alpha_k))$$ This would contradict (\ref{eqG1}).
\end{proof}

\subsection{Action of a one-cycle edge shift} \label{subs2}
Recall that the set of curves $\{\beta_1,\ldots,\beta_g,\gamma_1,\ldots,\gamma_{g-1}\}$ was defined in Figure \ref{fig_beta_gamma}, let $\gamma$ be one of these curves. Suppose that $\gamma$ is a one-cycle as in Figure \ref{fig_cycle} (a). In light of Equation (\ref{1_cycle_lift}), let us define
 \begin{align*}
\rt({E_e}) & = - \rt \left[\Big(t_{e}(\gamma)+A^{-1}\gamma Q_e^{-2}\Big)A^{-1} U(A^2 Q_e^2)^{-1}\right]  \\
\rt(E_e^{-1}) & =\quad \rt \left[\Big(A^3 \gamma Q_e^2+t_{e}(\gamma) \Big)A^{-1} U(A^2 Q_e^2)^{-1}F^{-1} \right] \end{align*}
where $F$ is defined in Equation (\ref{emb_1_cycle}). Notice that $\rt(F)$ is invertible by Lemma \ref{invF}, hence the second formula makes sense.

\noindent \textbf{Convention :} In the coming proofs, we sometimes drop the symbol $\rt$ when the notations are too cluttered.
\\

\begin{lemma} \label{1_cycle_rep} The following statements hold
\begin{enumerate}
\item[(a)] $\rt(Q_e) \rt(E_e) = -A  \rt(E_e) \rt( Q_e)$ and  $\rt(Q_e) \rt(E_e^{-1}) = -A ^{-1} \rt(E_e^{-1}) \rt( Q_e)$,
\item[(b)] $\rt(E_e) \rt(E_e^{-1}) = \mathrm{Id}_V$,
\item[(c)] $\rho(\gamma) = \rt(E_e)+\rt(E_e^{-1})\rt(F)$.
\end{enumerate}
\end{lemma}
\begin{proof}
Let us prove (a). Recall that $V_{e,k} = \mathrm{Ker}\Big(\rho(e)+(x_{e}^2 A^{2k+2} + x_{e}^{-2} A^{-2k-2}) \mathrm{Id}_V \Big)$ for $k \in \Z$. Moreover, it follow from $ e \gamma = A^{-2} \gamma e + A^{-1}(A^2-A^{-2}) t_{e}(\gamma)$ and $e t_{e}(\gamma) = A^2 t_{e}(\gamma) e - A(A^2-A^{-2}) \gamma$ that 
$$\rho(e) \rt(E_e) = \rt(E_e) \rt(-A^4 Q_e^2-A^{-4} Q_e^{-2})$$
Which implies that $\rt(E_e) V_{e,k} \subseteq V_{e,k+1}$ and in particular $\rt(Q_e) \rt(E_e) = -A  \rt(E_e) \rt( Q_e)$. The second part of (a) can be done with exact same strategy.

Let us prove (b). From (a), $U(A^2 Q_e^2)^{-1} \rt(E_e^{-1}) = \rt(E_e^{-1}) U(Q_e^2)$ so
$$\rt(E_e)\rt(E_e^{-1}) = -\Big(t_{e}(\gamma)+A^{-1}\beta Q_e^{-2}\Big)\Big(A^3 \gamma Q_e^2+t_{e}(\gamma) \Big) A^{-2} U(Q_e^2)^{-1} U(A^2 Q_e^2)^{-1}F^{-1}$$
Note that $-A^{2} U(Q_e^2) U(A^2 Q_e^2)F=-A^2 \Big(A^2 Q_e^4+A^{-2} Q_e^{-4} + f \Big)$. Thus it is enough to prove \begin{equation} \label{pic} \Big(t_{e}(\gamma)+A^{-1}\beta Q_e^{-2}\Big)\Big(A^3 \gamma Q_e^2+t_{e}(\gamma) \Big)=-A^2 \Big(A^2 Q_e^4+A^{-2} Q_e^{-4} + f \Big) \end{equation} Still from (a)
\begin{align*} 
& \Big(t_{e}(\gamma)+A^{-1}\gamma Q_e^{-2}\Big)\Big(A^3 \gamma Q_e^2+t_{e}(\gamma) \Big) =  A^3 t_{e}(\gamma) \gamma Q_e^2 + t_{e}(\gamma)^2+A^4 \gamma^2 + A \gamma t_{e}(\gamma) Q_e^{-2}
\end{align*}
Now we use $t_{e}(\gamma) \gamma  = A^{-1} e + A t_{\gamma}^{-2}(e)$ to get \begin{align} & \Big(t_{e}(\gamma)+A^{-1}\gamma Q_e^{-2}\Big)\Big(A^3 \gamma Q_e^2+t_{e}(\gamma) \Big)  = A^2(Q_e^2+Q_e^{-2})e + t_{e}(\gamma)^2+A^4 \gamma^2-A^2 t_{\gamma}^{-2}(e) e  \notag \\ 
& = -A^2 \Big(A^2 Q_e^4+A^{-2} Q_e^{-4} + f \Big) \notag
\end{align}
The last equality is obtained using $t_{\gamma}^{-2}(e)e = A^2 \gamma^2 + f -A^2-A^{-2} + A^{-2} t_{e}(\gamma)^2$ and $e = -(A^2 Q_e^2+A^{-2} Q_{e}^{-2})$. 

The proof of (c) follows from direct computation.
\end{proof}
\subsection{Action of a two-cycle edge shift} \label{subs3}
Let $\gamma \in \{\beta_1,\ldots,\beta_g,\gamma_1,\ldots,\gamma_{g-1}\}$ be a two-cycle as in Figure \ref{fig_cycle} (b). In light of Equations (\ref{2_cycle_lift1}), (\ref{2_cycle_lift2}), (\ref{2_cycle_lift3}) and (\ref{2_cycle_lift4}) let us define
\begin{align*}
&\rt(E_b E_c)  = \tilde{\rho}\left[ \Big( \gamma A^{-2} Q_b^{-2} Q_c^{-2}+ t_b(\gamma) A^{-1} Q_c^{-2} + t_c(\gamma) A^{-1} Q_b^{-2} +t_b t_c(\gamma) \Big) D^{-1} \right]\\
&\rt(E_b E_c^{-1})  = \tilde{\rho}\left[ -\Big( \gamma A^{2} Q_b^{-2} Q_c^{2}+ t_b(\gamma) A^{3} Q_c^{2} + t_c(\gamma) A^{-1} Q_b^{-2} +t_b t_c(\gamma) \Big) D^{-1} F_{1,-1}^{-1}\right] \\
&\rt(E_b^{-1} E_c) =  \tilde{\rho}\left[-\Big( \gamma A^{2} Q_b^{2} Q_c^{-2}+ t_b(\gamma) A^{-1} Q_c^{-2} + t_c(\gamma) A^{3} Q_b^{2} +t_b t_c(\gamma) \Big) D^{-1} F_{-1,1}^{-1} \right]\\
&\rt(E_b^{-1} E_c^{-1}) = \tilde{\rho}\left[ \Big( \gamma A^{6} Q_b^{2} Q_c^{2}+ t_b(\gamma) A^{3} Q_c^{2} + t_c(\gamma) A^{3} Q_b^{2} +t_b t_c(\gamma) \Big) D^{-1}F_{-1,-1}^{-1} \right]
\end{align*} where $D = A^2 U(A^2 Q_c^2) U(A^2 Q_b^2)$ and where $F_{1,-1},F_{-1,1},F_{-1,-1}$ were defined in Equation (\ref{emb_2_cycle}).

\begin{lemma} \label{2_cycle_rep}
The following statements hold :
\begin{enumerate}
\item[(a)] $\rt(Q_b) \rt(E_b^{\epsilon_1} E_c^{\epsilon_2}) = (-A)^{\epsilon_1} \rt(E_b^{\epsilon_1} E_c^{\epsilon_2})\rt(Q_b)$ for all $\epsilon_1,\epsilon_2 \in \{-1,1\}$,
\item[(b)] $\rt(Q_c) \rt(E_b^{\epsilon_1} E_c^{\epsilon_2}) = (-A)^{\epsilon_2} \rt(E_b^{\epsilon_1} E_c^{\epsilon_2})\rt(Q_c)$ for all $\epsilon_1,\epsilon_2 \in \{-1,1\}$,
\item[(c)] $\rt(E_b E_c) \rt(E_b^{-1} E_c^{-1}) = \mathrm{Id}_V$ and $\rt(E_b E_c^{-1}) \rt(E_b^{-1} E_c) = \mathrm{Id}_V$,
\item[(d)] $\rt(E_b^{\epsilon_1} E_c^{\epsilon_2}) \rt(E_b^{\epsilon_3} E_c^{\epsilon_4}) = \rt(E_b^{\epsilon_3} E_c^{\epsilon_4}) \rt(E_b^{\epsilon_1} E_c^{\epsilon_2})$ for all $\epsilon_1,\epsilon_2,\epsilon_3, \epsilon_4 \in \{-1,1\}$,
\item[(e)] $\rho(\gamma) = \rt(E_b E_c) + \rt(E_b E_c) \rt(F_{1,-1}) + \rt(E_b^{-1} E_c) \rt(F_{-1,1})+ \rt(E_b^{-1} E_c^{-1}) \rt(F_{-1,-1})$.
\end{enumerate}
\end{lemma}
\begin{proof}
For this proof, let us define for $\epsilon_1,\epsilon_2 \in \{-1,1\}$ $$X_{\epsilon_1,\epsilon_2} = \rt(E_b^{\epsilon_1} E_c^{\epsilon_2}) \rt(D F_{\epsilon_1,\epsilon_2})$$ where we set $F_{1,1} = 1$. 

The proof of (a) is very similar to the proof of Lemma \ref{1_cycle_rep} (a). This time,
$$\begin{array}{ll}
b \gamma &= A^{-2} \gamma b +A^{-1}(A^2-A^{-2}) t_{b}(\gamma)  \\
b t_c(\gamma) &= A^{-2} t_c(\gamma) b +A^{-1}(A^2-A^{-2}) t_{b}(t_c(\gamma)) \\
b t_b(\gamma) &= A^{2} t_b(\gamma)b -A(A^2-A^{-2}) \gamma  \\
b t_b t_c(\gamma) &= A^{2} t_b t_c(\gamma)b -A(A^2-A^{-2}) t_c(\gamma) 
\end{array}$$
implies that $\rt(b) \rt(E_b^{\epsilon_1} E_c^{\epsilon_2}) = \rt(E_b^{\epsilon_1} E_c^{\epsilon_2}) \rt(-A^{2+2\epsilon_1} Q_b^2-A^{2-2\epsilon_1} Q_b^{-2})$ for $\epsilon,\epsilon_2 = \pm 1$.
The proof of (b) is exactly the same as (a). 

Let us prove (c).
$$\rt(E_b E_c^{\epsilon}) \rt(E_b^{-1} E_c^{-\epsilon}) = X_{1,\epsilon} X_{-1,-\epsilon}  (A^4 \hat{F}_{1,\epsilon} F_{-1,-\epsilon} \hat{D} D)^{-1}$$
with $$A^4 \hat{F}_{1,\epsilon} F_{-1,-\epsilon} \hat{D} D = A^4(A^{2\epsilon}Q_b^2 Q_c^{2\epsilon}+A^{-2\epsilon} Q_b^{-2} Q_c^{-2\epsilon}+a)(A^{2\epsilon}Q_b^2 Q_c^{2\epsilon}+A^{-2\epsilon} Q_b^{-2} Q_c^{-2\epsilon}+a')$$
Let us prove that $$X_{1,\epsilon} X_{-1,-\epsilon} =  A^4 \hat{F}_{1,\epsilon} F_{-1,-\epsilon} \hat{D} D$$
We fix $\epsilon \in \{-1,1\}$. 
Let $x = A^{-1} \gamma Q_b^{-2}+t_b(\gamma)$ and $y = A^3 \gamma Q_b^2+t_b(\gamma)$, we have
\begin{align*}
& X_{1,\epsilon}  = \epsilon(x A^{1-2\epsilon}Q_c^{-2\epsilon}+t_c(x)) \\
& X_{-1,-\epsilon}  = \epsilon(yA^{1+2\epsilon}Q_c^{2\epsilon}+t_c(y))
\end{align*}
$$X_{1,\epsilon} X_{-1,-\epsilon} = A^4 xy +t_c(xy)+A^2(A^{2\epsilon -1} t_c(x)yQ_c^{2\epsilon}+A^{-2\epsilon+1}xt_c(y)Q_c^{-2\epsilon})$$
Notice that the computations done in Lemma \ref{1_cycle_rep} can be repeated (more precisely Equation \ref{pic}) and give 
\begin{equation} \label{u1}
xy = -A^2(A^2Q_b^4+A^{-2} Q_b^{-4}+\delta_b)\end{equation}
\begin{equation} \label{u2} t_c(xy) = -A^2(A^2Q_b^4+A^{-2} Q_b^{-4}+t_c(\delta_b)) \end{equation}

We compute $A^{2\epsilon -1} t_c(x)yQ_c^{2\epsilon}+A^{-2\epsilon+1}xt_c(y)Q_c^{-2\epsilon} = A^4 z+t_b(z)+A^2 z'$ where $z = A^{2 \epsilon-1}t_c(\gamma) \gamma Q_c^{2\epsilon}+A^{1-2\epsilon}\gamma t_c(\gamma) Q_c^{-2\epsilon}$ and $$z' = A^{2\epsilon} t_c t_b(\gamma) \gamma Q_b^2 Q_c^{2\epsilon}+A^{-2\epsilon}  \gamma t_c t_b(\gamma) Q_b^{-2} Q_c^{-2 \epsilon}+ t_b(\gamma)t_c(\gamma)\Big[A^{2\epsilon-2}Q_b^{-2}Q_c^{2\epsilon}+A^{-2\epsilon+2}Q_b^{2}Q_c^{-2\epsilon}\Big]$$
Hence \begin{equation} \label{u} X_{1,\epsilon} X_{-1,-\epsilon} = A^4xy+t_c(xy) + A^2 (A^4 z+t_b(z)+A^2 z') \end{equation}

We compute using skein relations :
\begin{equation} \label{u3} z = -A^2 \gamma^2- A^{-2} (t_c(\gamma))^2 -\delta_c +(A^2+A^{-2})+c(A^{2-2\epsilon}
Q_c^2+A^{-2+2 \epsilon}
Q_c^{-2}) \end{equation}

\begin{equation} \label{u4} t_b(z) = -A^2 (t_b(\gamma))^2- A^{-2} (t_bt_c(\gamma))^2 -t_b(\delta_c) +(A^2+A^{-2})+c(A^{2-2\epsilon}
Q_c^2+A^{-2+2 \epsilon}
Q_c^{-2}) \end{equation}

\begin{equation} \label{u5} z' = t_c(\gamma)t_b(\gamma) c b + (a+a')(A^{2\epsilon}Q_b^2 Q_c^{2\epsilon}+A^{-2\epsilon} Q_b^{-2} Q_c^{-2\epsilon}) +cb (A^{2\epsilon-2}Q_b^2 Q_c^{2\epsilon}+A^{2-2\epsilon}Q_b^{-2} Q_c^{-2\epsilon}) \end{equation}
The computation of $ t_c(\gamma) t_b(\gamma) cb$ gives 

\noindent $ A^4 \gamma^2+A^2(\delta_b+\delta_c)+(b^2+c^2-(A^2+A^{-2})^2+aa'+(t_b(\gamma))^2+(t_c(\gamma))^2)+A^{-2}(t_b(\delta_c)+t_c(\delta_b))+A^{-4}(t_bt_c(\gamma))^2$

Combining (\ref{u1}),(\ref{u2}),(\ref{u3}),(\ref{u4}),(\ref{u5}) in (\ref{u}) we get 

$$X_{1,\epsilon} X_{-1,-\epsilon}  
= A^4(A^{2\epsilon}Q_b^2 Q_c^{2\epsilon}+A^{-2\epsilon} Q_b^{-2} Q_c^{-2\epsilon}+a)(A^{2\epsilon}Q_b^2 Q_c^{2\epsilon}+A^{-2\epsilon} Q_b^{-2} Q_c^{-2\epsilon}+a')$$

Let us prove (d). By (c) it is enough to prove that $$\rt(E_b E_c) \rt(E_b E_c^{-1}) = \rt(E_b E_c^{-1}) \rt(E_b E_c)$$
We have \begin{align*}
\rt(E_b E_c) \rt(E_b E_c^{-1}) &= X_{1,1} X_{1,-1} \Big(A^4 U(A^2Q_b^2)U(A^2Q_c^2)U(A^4 Q_b^2) U(Q_c^2) F_{1,-1}\Big)^{-1} \\
\rt(E_b E_c^{-1}) \rt(E_b E_c) &= X_{1,-1} X_{1,1} \Big(A^4 U(A^2Q_b^2)U(A^2Q_c^2)U(A^4 Q_b^2) U(Q_c^2) F_{1,-1}\Big)^{-1}
\end{align*}
Thus it is enough to prove $X_{1,1}X_{1,-1} = X_{1,-1}X_{1,1}$. This reduces to \begin{align*}
& A^4 x^2+t_c(x^2)+Axt_c(x)Q_c^{-2}+A^3t_c(x)xQ_c^{2} = A^4 x^2+t_c(x^2) +A^5xt_c(x)Q_c^{2}+A^{-1}t_c(x)xQ_c^{-2} \\
&\iff Axt_c(x)Q_c^{-2}+A^3t_c(x)xQ_c^{2} =  A^5xt_c(x)Q_c^{2}+A^{-1}t_c(x)xQ_c^{-2} \\
& \iff Axt_c(x)(A^2 Q_c^2-A^{-2} Q_c^{-2}) = A^{-1} t_c(x)x(A^2 Q_c^2-A^{-2} Q_c^{-2}) \\
& \iff Axt_c(x) = A^{-1} t_c(x)x\end{align*}
The last equivalence is obtained because $U(A^2 Q_c^2)$ is invertible. Let us expand $Axt_c(x) - A^{-1} t_c(x)x$ remembering that $Q_b x = A x Q_b$ and prove it is zero :
\begin{align*}
& Axt_c(x) - A^{-1} t_c(x)x = A^{-4}\Big[A\gamma t_c(\gamma)-A^{-1}t_c(\gamma) \gamma \Big]Q_b^{-4} + \\
& A^{-2}\Big[\gamma t_c(\gamma)-t_c(\gamma) \gamma  +(A^2-A^{-2})t_b(\gamma)t_c(\gamma)\Big]Q_b^{-2} + t_b\Big(A \gamma t_c(\gamma)-A^{-1} t_c(\gamma) \gamma \Big)
\end{align*}
Expanding $\gamma t_c(\gamma)$ and $t_c(\gamma) \gamma$ using skein relations, this expression reduces to 
\begin{align*}
& (A^2-A^{-2}) (A^{-4}cQ_b^{-4}+A^{-2}bc Q_b^{-2} +c) \\
& = (A^2-A^{-2})\Big[c(A^2 Q_b^2+A^{-2} Q_b^{-2})+bc \Big] Q_b^{-2} A^{-2} \\
& = (A^2-A^{-2})\Big[-cb+bc \Big] Q_b^{-2} A^{-2} \\
& = 0
\end{align*}

\end{proof}

\subsection{Action of the square of a separating edge shift} \label{subs4}
Let $\gamma$ be a separating edge curve as shown in Figure \ref{fig_cycle} (c). Again inspired by Equations (\ref{sep_lift1}) and (\ref{sep_lift2}) let \begin{align*}
 \rt(E_c^2) & = -\tilde{\rho} \left[\Big(\gamma Q_c^{-2} + \tau - \dfrac{A^{-2} \delta_1 Q_c^{-2}-A^2 \delta_2}{U(A^4 Q_c^2)}\Big)A^{-2} U(A^2 Q_c^2)^{-1} G_2^{-1} \right]  \\
 \rt(E_c^{-2}) & = \tilde{\rho} \left[\Big(\gamma A^2 Q_c^2 + A^{-2} \tau + \dfrac{ \delta_1 Q_c^2-\delta_2}{U(Q_c^2)} \Big) U(A^2 Q_c^2)^{-1} G_{-2}^{-1} \right]   
\end{align*}

\begin{lemma} \label{sep_rep}The following statements hold :
\begin{enumerate}
\item[(a)] $\tilde{\rho}(Q_c) \rt(E_c^2) = A^2 \rt(E_c^2) \tilde{\rho}(Q_c)$ and $\tilde{\rho}(Q_c) \rt(E_c^{-2}) = A^{-2} \rt(E_c^{-2}) \tilde{\rho}(Q_c)$,
\item[(b)] $\rt(E_c^2) \rt(E_c^{-2})= \mathrm{Id}_V$.
\item[(c)] $\rho(\gamma) = \rt(E_c^2) \rt(G_2)+ \rt(G_0) + \rt(E_c^{-2}) \rt(G_{-2})$.
\end{enumerate}

\end{lemma}
\begin{proof}
(a) Using skein relations we have 
\begin{align*}
c \gamma & = A^{-4} \gamma c + A^{-2}(A^4-A^{-4}) \tau + A^{-2} (A^2-A^{-2}) (d_1 d_3 + d_2 d_4) \\
c \tau & = A^{4} \tau c - A^{2}(A^4-A^{-4}) c - A^{2} (A^2-A^{-2}) (d_1 d_2 + d_3 d_4) 
\end{align*}
Then a computation implies $$c E_c^{\pm 2} = E_c^{\pm 2} (-A^{2 \pm 4} Q_c^2 - A^{-(2 \pm 4)} Q_c^{-2})$$ which implies the desired equalities.

(b) Let $Y_2 = \rt(E_c^2) A^{2} U(A^2 Q_c^2)^{1} G_2^{1}$ and $Y_{-2} = \rt(E_c^{-2}) U(A^2 Q_c^2) G_{-2}$. Using (a), we have $\rt(E_c^2) \rt(E_c^{-2}) = Y_2 Y_{-2} \Big( A^2 U(A^{-2} Q_c^2) \hat{G_2} U(A^2 Q_c^2) G_{-2} \Big)^{-1} $ where $$\hat{G_2} = -U(A^{2} Q_{d_1} Q_{d_4} Q_{c}^{-1}) U(A^{2} Q_{d_2} Q_{d_3} Q_{c}^{-1})$$ Let us prove that $Y_2 Y_{-2} = A^2 U(A^{-2} Q_c^2) \hat{G_2} U(A^2 Q_c^2) G_{-2} $. A brute force computation gives 
\begin{align*}
& A^2 U(A^{-2} Q_c^2) \hat{G_2} U(A^2 Q_c^2) G_{-2} \\
& = -A^2 \Big(-T^4+\delta_3 T^3 +(8-\Delta) T^2+(\delta_1 \delta_2-4 \delta_3) T + (4\Delta -16 - \delta_1^2 - \delta_2^2) \Big) U(Q_c^2)^{-2}
\end{align*}
where $\Delta = d_1^2+d_2^2+d_3^2+d_4^2+d_1d_2d_3d_4$ and $T = Q_c^2+Q_c^{-2}$. (Note that this egality is equivalent to an equality of two Laurent polynomials in $A$ and the variables $Q_e,$ hence can be checked using Sage for example). Therefore it is enough to prove that $$-Y_2 Y_ {-2} = A^2 \Big(-T^4+\delta_3 T^3 +(8-\Delta) T^2+(\delta_1 \delta_2-4 \delta_3) T + (4\Delta -16 - \delta_1^2 - \delta_2^2) \Big) U(Q_c^2)^{-2}$$
Let us prove this :
\begin{align*}
& -Y_2 Y_ {-2} = \left(\gamma Q_c^{-2} + \tau - \dfrac{A^{-2} \delta_1 Q_c^{-2}-A^2 \delta_2}{U(A^4 Q_c^2)}\right) \left(\gamma A^2 Q_c^2 + A^{-2} \tau + \dfrac{ \delta_1 Q_c^2-\delta_2}{U(Q_c^2)}\right) \\
& = A^6 \gamma^2 + A^2 \gamma \tau Q_c^{-2} +A^4 \gamma \delta_1 U(Q_c^2)^{-1} - A^4 \gamma \delta_2 Q_c^{-2} U(Q_c^2)^{-1} \\
& + A^2 \tau \gamma Q_c^2+A^{-2} \tau^2 + \tau \delta_1 Q_c^2 U(Q_c^2)^{-1}-\tau \delta_2 U(Q_c^2)^{-1}\\
& -A^4 \delta_1 \gamma U(Q_c^2)^{-1}-\delta_1 \tau Q_c^{-2} U(Q_c^2)^{-1}-A^2 \delta_1^2 U(Q_c^2)^{-2}+ A^2 \delta_1 \delta_1 Q_c^{-2} U(Q_c^2)^{-2} \\
& +A^4 \delta_2 \gamma Q_c^2+\delta_2 \tau U(Q_c^2)^{-1} + A^2 \delta_1 \delta_2 Q_c^2 U(Q_c^2)^{-2}- A^2 \delta_2^2 U(Q_c^2)^{-2}\\
& = A^4 \delta_2 \gamma + \tau \delta_1 + \big[A^2 \delta_1 \delta_2 (Q_c^2+Q_c^{-2})-A^2(\delta_1^2+\delta_2^2)\big]U(Q_c^2)^{-2} \\
& + A^6 \gamma^2+A^2(\gamma \tau Q_c^{-2}+\tau \gamma Q_c^2)+A^{-2} \tau^2
\end{align*}
The term $\gamma \tau$ can reduced using skein relations to $A^2 c + \delta_3+A^{-2} \varphi$, where $\delta_3 = d_1 d_4+d_2 d_3$ and $\varphi$ is 
$$\includegraphics[scale = 0.3]{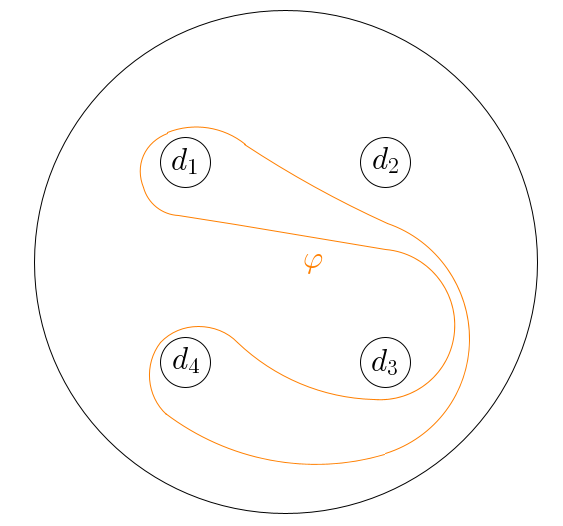}$$
Hence $\gamma \tau Q_c^{-2}+\tau \gamma Q_c^2 =(A^{-2} Q_c^2+A^2 Q_c^{-2})c+ \delta_3(Q_c^2+Q_c^{-2})-\varphi c $. The expansion of $\varphi c$ gives 
$$\varphi c = \Delta -(A^2+A^{-2})^2+A^4 \gamma^2+A^2 \delta_2 \gamma+A^{-2} \delta_1 \tau+A^{-4} \tau^2$$ we recall that $\Delta = d_1^2+d_2^2+d_3^2+d_4^2+d_1d_2d_3d_4$. Now plugging this back to the expression of $-Y_2Y_{-2},$ we get after simplifications
\begin{align*}
& -Y_2 Y_ {-2} = \left(\gamma Q_c^{-2} + \tau - \dfrac{A^{-2} \delta_1 Q_c^{-2}-A^2 \delta_2}{U(A^4 Q_c^2)}\right) \left(\gamma A^2 Q_c^2 + A^{-2} \tau + \dfrac{ \delta_1 Q_c^2-\delta_2}{U(Q_c^2)}\right) \\
& = A^2 \Big( \delta_1 \delta_2 (Q_c^2+Q_c^{-2})U(Q_c^2)^{-2} - (\delta_1^2+\delta_2^2) U(Q_c^2)^{-2} + \delta_3 (Q_c^2+Q_c^{-2}) - \Delta - U(Q_c^2)^2\Big) \\
& = A^2 \Big( \delta_1 \delta_2 (Q_c^2+Q_c^{-2}) - (\delta_1^2+\delta_2^2)  + \delta_3 (Q_c^2+Q_c^{-2})U(Q_c^2)^{2} - \Delta U(Q_c^2)^{2} - U(Q_c^2)^{4} \Big)U(Q_c^2)^{-2} \\
& = A^2 \Big(-T^4+\delta_3 T^3 +(8-\Delta) T^2+(\delta_1 \delta_2-4 \delta_3) T + (4\Delta -16 - \delta_1^2 - \delta_2^2) \Big) U(Q_c^2)^{-2}
\end{align*}
where we recall that $T = Q_c^2+Q_c^{-2}.$ The last equality is obtained from the identity $U(Q_c^2)^2 = T^2-4$.

Point (c) is derived from a direct computation.
\end{proof}

\subsection{Commutation between the square of an edge and a one-cycle}\label{sec:com1cse}
Consider the following portion of the graph 
$$\includegraphics[scale = 0.35]{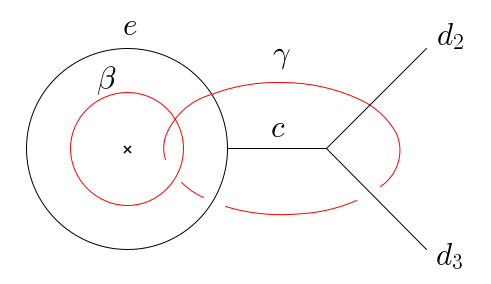}$$ 
Recall that $\rt(E_c^2)  = \tilde{\rho} \left[Y_2 A^{-2} U(A^2 Q_c^2)^{-1} G_2^{-1} \right]$ where $G_2  =  -U(Q_{e}^2 Q_{c}^{-1}) U(Q_{d_2} Q_{d_3} Q_{c}^{-1})$ and 
where $$-Y_2 = \gamma Q_c^{-2} + \tau - \dfrac{A^{-2} \delta_1 Q_c^{-2}-A^2 \delta_2}{U(A^4 Q_c^2)}$$
with $\delta_1 = e d_3+e d_2 $, $\delta_2 = e d_2+e d_3 $ and $\tau$ being the following curve :
$$\includegraphics[scale = 0.35]{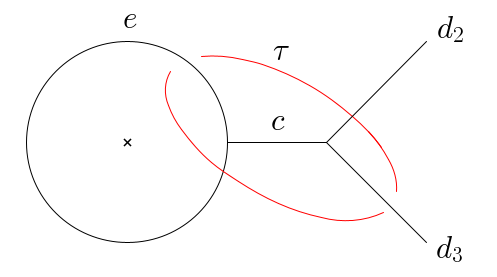}$$ 
Recall also that $\rt(E_e) = -\tilde{\rho} \left[\Big(t_{e}(\beta)+A^{-1}\beta Q_e^{-2}\Big)A^{-1} U(A^2 Q_e^2)^{-1} \right]$

\begin{lemma} \label{com1cse}
$\rt(E_e) \rt(E_c^2) = \rt(E_c^2) \rt(E_e)$.
\end{lemma}
\begin{proof}
Let us first simplify the expression of $\rt(E_c^2)$. A computation shows that
$$\rt(E_c^2) = - \left[\Big(\gamma-t_c(\gamma)+(\bar{\tau}-\tau) Q_c^{-2}\Big)A^{-2} U(A^4 Q_c^2)^{-1} U(A^2 Q_c^2)^{-1} G_2^{-1} \right]$$ where $\bar{\tau}=t_c^{-1}(\tau)$ is the following curve 
$$\includegraphics[scale = 0.35]{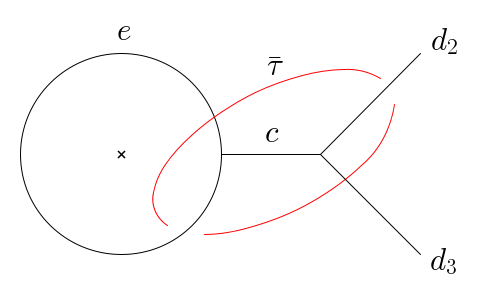}$$ 
Let us define $Y_2' = \gamma-t_c(\gamma)+(\bar{\tau}-\tau) Q_c^{-2}$, $\varphi = \gamma-t_c(\gamma)$ and $\psi = \bar{\tau}-\tau$. Using $A$-commutation we see that $\rt(E_e) \rt(E_c^2) = \rt(E_c^2) \rt(E_e)$ is equivalent to 
\begin{equation} \label{eq11} X_1 Y_2'(A^2 Q_e^2 Q_c^{-1}-A^{-2} Q_e^{-2}Q_c) = Y_2' X_1 (Q_e^2 Q_c^{-1}-Q_e^{-2} Q_c) 
\end{equation} where $X_1=t_{e}(\beta)+A^{-1}\beta Q_e^{-2}$. For two elements $x,y$ we define $[x,y]_A = Axy-A^{-1}yx$. Proving Equality $(\ref{eq11})$ is equivalent to prove that $$\Big( A[X_1,Y_2']_A Q_e^2 +A^{-1} [Y_2',X_1]_AQ_e^{-2} Q_c^2 \Big) Q_c^{-1} = 0$$
Now let us expand the expression $\mathcal{E} = A[X_1,Y_2']_A Q_e^2 +A^{-1} [Y_2',X_1]_AQ_e^{-2} Q_c^2 $ using the fact that $X_1$ commutes with $Q_c$ and $Y_2'$ commutes with $Q_e$ : 
\begin{align*}
& A[X_1,Y_2']_A Q_e^2 +A^{-1} [Y_2',X_1]_AQ_e^{-2} Q_c^2 \\
&= A\Big( [t_e(\beta),\varphi]_A+[t_e(\beta),\psi]_AQ_c^{-2}+A^{-1}[\beta,\varphi]_A Q_e^{-2} +A^{-1}[\beta,\psi]_A Q_c^{-2} Q_e^{-2} \Big) Q_e^2 \\
&+A^{-1}\Big([\varphi,t_e(\beta)]_A+A^{-1}[\varphi,\beta]_A Q_e^{-2}+[\psi,t_e(\beta)]_A Q_c^{-2} +A^{-1}[\psi,\beta]_A Q_c^{-2} Q_e^{-2} \Big)Q_e^{-2} Q_c^2 \\
&=A\Big( [t_e(\beta),\varphi]_A Q_e^2+[t_e(\beta),\psi]_AQ_c^{-2}Q_e^2+A^{-1}[\beta,\varphi]_A  +A^{-1}[\beta,\psi]_A Q_c^{-2}  \Big) \\
&+A^{-1}\Big([\varphi,t_e(\beta)]_A Q_e^{-2} Q_c^2+A^{-1}[\varphi,\beta]_A Q_e^{-4} Q_c^2+[\psi,t_e(\beta)]_A Q_e^{-2} +A^{-1}[\psi,\beta]_A Q_e^{-4} \Big) \\
\end{align*}
Using that $Q_e^{-2} = -A^2 e -A^4 Q_e^2$, $Q_c^{-2} = -A^2 c -A^4 Q_c^2$ and $Q_e^{-4} = A^4 e^2+A^6 e Q_e^2-A^4$ we get $$\mathcal{E} = C_1 Q_c^2 Q_e^2+C_2 Q_e^2+C_3 Q_c^2+C_4$$ where 
\begin{align*}
C_1 & = -A^5[t_e(\beta),\psi]_A-A^3[\varphi,t_e(\beta)]_A+A^4[\varphi,\beta]_Ae \\
C_2 & = A[t_e(\beta),\varphi]_A-A^3[t_e(\beta),\psi]_Ac-A^3[\psi,t_e(\beta)]_A+A^4[\psi,\beta]_Ae \\
C_3 & =  -A^4[\beta,\psi]_A-A[\varphi,t_e(\beta)]_Ae+A^2[\varphi,\beta]_Ae^2-A^2[\varphi,\beta]_A \\
C_4 & = [\beta,\varphi]_A-A^2[\beta,\psi]_Ac-A[\psi,t_e(\beta)]_Ae+A^2[\psi,\beta]_Ae^2-A^2[\psi,\beta]_A
\end{align*}
The elements $C_1,C_2,C_3$ and $C_4$ are skein elements that can be computed using skein relations. After a straightforward computation we get $C_1=C_2=C_3=C_4=0$ which concludes that $\mathcal{E} = 0$. Some more details on those computations are given in the appendix.
\end{proof}
\subsection{Commutation between the square of an edge and a two-cycle}\label{sec:com2cse}
Consider the following portion of the graph 
$$\includegraphics[scale = 0.35]{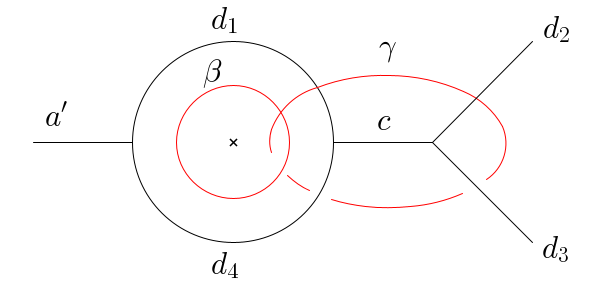}$$ 
Using the proof of Lemma \ref{com1cse} we have $$\rt(E_c^2) = - \tilde{\rho} \left[Y_2'A^{-2} U(A^4 Q_c^2)^{-1} U(A^2 Q_c^2)^{-1} G_2^{-1} \right]$$ with $Y_2' = \gamma-t_c(\gamma)+(\bar{\tau}-\tau) Q_c^{-2}$. Recall also that we set $\varphi = \gamma-t_c(\gamma)$ and $\psi = \bar{\tau}-\tau$. For the two-cycle $\beta$ we have
\begin{align*}
&\rt(E_{d_1} E_{d_4})  = \tilde{\rho}\left[ \Big( \beta A^{-2} Q_{d_1}^{-2} Q_{d_4}^{-2}+ t_{d_1}(\beta) A^{-1} Q_{d_4}^{-2} + t_{d_4}(\beta) A^{-1} Q_{d_1}^{-2} +t_{d_1} t_{d_4}(\beta) \Big) D^{-1} \right]\\
&\rt(E_{d_1} E_{d_4}^{-1}) \tilde{\rho}(F_{1,-1}) = \tilde{\rho}\left[ -\Big( \beta A^{2} Q_{d_1}^{-2} Q_{d_4}^{2}+ t_{d_1}(\beta) A^{3} Q_{d_4}^{2} + t_{d_4}(\beta) A^{-1} Q_{d_1}^{-2} +t_{d_1} t_{d_4}(\beta) \Big) D^{-1}\right] 
\end{align*} where $D = A^2 U(A^2 Q_{d_1}^2) U(A^2 Q_{d_4}^2)$ and $F_{1,-1} = -\dfrac{U(Q_{a'} Q_{d_4} Q_{d_1}^{-1}) U(Q_{c} Q_{d_4} Q_{d_1}^{-1})}{U(A^2 Q_{d_4}^2)U(Q_{d_4}^2)}$.
\begin{lemma} \label{com2cse} The following statements hold :
\begin{enumerate}
\item[(a)] $\rt(E_{d_1} E_{d_4}) \rt(E_c^2) = \rt(E_c^2) \rt(E_{d_1} E_{d_4})$,
\item[(b)] $\rt(E_{d_1} E_{d_4}^{-1}) \rt(E_c^2) = \rt(E_c^2) \rt(E_{d_1} E_{d_4}^{-1})$.
\end{enumerate}
\end{lemma}

\begin{proof} To make the notation less cluttered, we write $Q_4$ for $Q_{d_4}$ and $Q_1$ for $Q_{d_1}$.

Let us prove (a). First let us set $X = A^{-1} \beta Q_{1}^{-2}+t_{d_1}(\beta)$ so that $$\rt(E_{d_1} E_{d_4})  = \Big( A^{-1} X Q_{4}^{-2} + t_{d_4}(X) \Big) D^{-1} $$ Using known commutation relations we have that $\rt(E_{d_1} E_{d_4}) \rt(E_c^2) = \rt(E_c^2) \rt(E_{d_1} E_{d_4})$ is equivalent to $$\mathcal{E} = A^{-1}\Big(A[X,Y_2']_AQ_1^{2}Q_4^2+A^{-1}[Y_2',X]Q_c^2\Big)Q_4^{-2} +t_{d_4}\Big(A[X,Y_2']_AQ_1^{2}Q_4^2+A^{-1}[Y_2',X]Q_c^2\Big)=0$$
If we set $\mathcal{U} = A[X,Y_2']_AQ_1^{2}Q_4^2+A^{-1}[Y_2',X]Q_c^2$, the expression of $\mathcal{E}$ is simply $A^{-1} \mathcal{U} Q_{4}^{-2} + t_{d_4}(\mathcal{U})$. Let us first compute $\mathcal{U}$ by expanding the expression and remembering that $Q_{1}^{-2} = -A^{4} Q_{1}^2-A^2 d_1$ and $Q_{c}^{-2} = -A^{4} Q_{c}^2-A^2 c$  : \begin{align*}
& \mathcal{U} = A[X,Y_2']_AQ_1^{2}Q_4^2+A^{-1}[Y_2',X]Q_c^2 \\
& = -A^5 [t_{d_1}(\beta),\psi]_A Q_c^2 Q_1^2 Q_4^2  \\
& - \Big( A^3 [t_{d_1}(\beta),\psi]_A c +A[t_{d_1}(\beta),\varphi]_A \Big)Q_1^2 Q_4^2 -A^4[\beta,\psi]_A Q_c^2 Q_4^2  - A^2 [\varphi,\beta]_A Q_1^2 Q_c^2 \\
& +\Big( [\beta,\varphi]_A - A^2 [\beta,\psi]_A c\Big) Q_4^2 -A^2[\psi,\beta]_A Q_1^2  + \Big(A^{-1}[\varphi,t_{d_1}(\beta)]_A-[\varphi,\beta]_A d_1 \Big)Q_c^2 \\
& +\Big(A^{-1}[\psi,t_{d_1}(\beta)]_A-[\psi,\beta]_A d_1\Big)
\end{align*}
Using the same technique we get 
$$ A^{-1} \mathcal{U} Q_4^{-2} +t_{d_4}(\mathcal{U}) = \sum_{\epsilon_1,\epsilon_2,\epsilon_3 = 0,1} C_{\epsilon_1,\epsilon_2,\epsilon_3} Q_1^{2 \epsilon_1} Q_4^{ 2 \epsilon_2} Q_c^{2 \epsilon_3}
$$ where 
\begin{align*}
 C_{1,1,1} &= -A^5 [t_{d_4}t_{d_1}(\beta),\psi]_A+A^5[\varphi,\beta]_A \\
 C_{1,1,0} &= A^5 [\psi,\beta]_A-A^3[t_{d_4}t_{d_1}(\beta),\psi]_Ac+A[t_{d_4}t_{d_1}(\beta),\varphi]_A \\
 C_{1,0,1} &= -A^2[\varphi,t_{d_4}(\beta)]_A+A^3[\varphi,\beta]_Ad_4-A^4 [t_{d_1}(\beta),\psi]_A \\
 C_{0,1,1} &= -A^2 [\varphi,t_{d_1}(\beta)]_A+A^3[\varphi,\beta]_A d_1-A^4 [t_{d_4}(\beta),\psi]_A \\
 C_{0,0,1} &= -A^3[\beta,\psi]_A-[\varphi,t_{d_1}(\beta)]_A d_4+A [\varphi,\beta]_A d_1 d_4+A^{-1}[\varphi,t_{d_4}t_{d_1}(\beta)]_A-[\varphi,t_{d_4}(\beta)]_A d_1 \\
 C_{0,1,0} &= A^3[\psi,\beta]_A d_1+[t_{d_4}(\beta),\varphi]_A-A^2[t_{d_4}(\beta),\psi]_A c-A^2 [\psi,t_{d_1}(\beta)]_A \\
 C_{1,0,0} &= -A^2[t_{d_1}(\beta),\psi]_A c+[t_{d_1}(\beta),\varphi]_A-A^2[\psi,t_{d_4}(\beta)]_A+A^3[\psi,\beta]_A d_4 \\
 C_{0,0,0} &= -[\psi,t_{d_4}(\beta)]_A d_1+ A^{-1}[\psi,t_{d_4}t_{d_1}(\beta)]_A+A^{-1} [\beta,\varphi]_A \\
 & -A [\beta,\psi]_Ac+A [\psi,\beta]_A d_1 d_4-[\psi,t_{d_1}(\beta)]_A d_4
\end{align*}
Here again, these elements are in the skein algebra of the surface hence can be computed using skein relations. A long but straightforward computation shows that $C_{\epsilon_1,\epsilon_2,\epsilon_3} = 0$ for all $\epsilon_1,\epsilon_2,\epsilon_3 \in \{0,1\}$. We refer to the appendix for more details. Thus $\mathcal{E} = A^{-1} \mathcal{U} Q_4^{-2} +t_{d_4}(\mathcal{U}) = 0$.

Let us prove (b). The proof goes as in (a) and we use the same notations. In this case $\rt(E_{d_1} E_{d_4}^{-1}) \rt(E_c^2) = \rt(E_c^2) \rt(E_{d_1} E_{d_4}^{-1})$ is equivalent to $$\mathcal{E} = A^{3}\Big(A[Y_2',X]_A Q_c^{2}Q_4^2+A^{-1}[X,Y_2']Q_1^2\Big)Q_4^{2} +t_{d_4}\Big(A[Y_2',X]_A Q_c^{2}Q_4^2+A^{-1}[X,Y_2']Q_1^2\Big)=0$$ Expanding this expression using $Q_{1}^{-2} = -A^{4} Q_{1}^2-A^2 d_1$, $Q_{c}^{-2} = -A^{4} Q_{c}^2-A^2 c$ and $Q_4^4 = -A^{-2} Q_4^2 d_4-A^{-4}$ we get :
$$ \mathcal{E} = \sum_{\epsilon_1,\epsilon_2,\epsilon_3 = 0,1} D_{\epsilon_1,\epsilon_2,\epsilon_3} Q_1^{2 \epsilon_1} Q_4^{ 2 \epsilon_2} Q_c^{2 \epsilon_3}
$$ where 
\begin{align*}
 D_{1,1,1} &= A^5 [\varphi,\beta]_A d_4-A^6[t_{d_1}(\beta),\psi]_A-A^4[\varphi,t_{d_4}(\beta)]_A \\
 D_{1,1,0} &= A^5 [\psi,\beta]_A d_4 +A^2[t_{d_1}(\beta),\varphi]_A-A^4[t_{d_1}(\beta),\psi]_Ac-A^4[\psi,t_{d_4}(\beta)]_A \\
 D_{1,0,1} &= A^3[\varphi,\beta]_A-A^3[t_{d_1}t_{d_4}(\beta),\psi]_A \\
  D_{0,1,1} &= -A^2 [\varphi,t_{d_1}(\beta)]_A d_4+A^3[\varphi,\beta]_A d_1 d_4 -A^5 [\beta,\psi]_A +A[\varphi,t_{d_4}t_{d_1}(\beta)]_A -A^2[\varphi,t_{d_4}(\beta)]_A d_1 \\
 D_{0,0,1} &= -[\varphi,t_{d_1}(\beta)]_A+A [\varphi,\beta]_A d_1-A^2 [t_{d_4}(\beta),\psi]_A  \\
 D_{0,1,0} &= -A^2[\psi,t_{d_1}(\beta)]_A d_4-A^3[\beta,\psi]_Ac+A[\beta,\varphi]_A +A^3 [\psi,\beta]_A d_1 d_4 \\
 &+A[\psi,t_{d_1}t_{d_4}(\beta)]_A-A^2[\psi,t_{d_4}(\beta)]_A d_1 \\
 D_{1,0,0} &= A^3[\psi,\beta]_A +A^{-1}[t_{d_1}t_{d_4}(\beta),\varphi]_A-A[t_{d_1}t_{d_4}(\beta),\psi]_A c \\
 D_{0,0,0} &= A [\psi,\beta]_A d_1-[\psi,t_{d_1}(\beta)]_A-[t_{d_4}(\beta),\psi]_A c+ A^{-2}[t_{d_4}(\beta),\varphi]_A
\end{align*}
Notice that $D_{1,1,1}=A^2C_{1,0,1},$ $D_{1,1,0}=A^2C_{1,0,0},$ $D_{1,0,1}=A^{-2}C_{1,1,1},$ $D_{0,1,1}=A^2C_{0,0,1},$ $D_{0,0,1}=A^{-2}C_{0,1,1},$ $D_{0,1,0}=A^2C_{0,0,0},$ $D_{1,0,0}=A^{-2}C_{1,1,0},$ and $D_{0,0,0}=A^{-2}C_{0,1,0}.$
Hence these elements are again all vanishing.
\end{proof}

\subsection{Proof of Theorem \ref{thm:lift_rep}} \label{proof_lift_rep}

Lemma \ref{gen_t} gives us a list of generators for $\mathcal{A}_A(\Gamma)^0$. Subsections \ref{subs1}, \ref{subs2}, \ref{subs3} and \ref{subs4}, define $\rt$ on these generators. Lemmas \ref{1_cycle_rep} (a) (b), \ref{2_cycle_rep} (a) (b) (c) (d), \ref{sep_rep} (a) (b), \ref{com1cse}, \ref{com2cse} insure that $\rt$ preserve the relations between the generators and therefore defines a representation $\rt : \mathcal{A}_A(\Gamma)^0 \to \mathrm{End}(V)$. Now Lemmas \ref{1_cycle_rep} (c), \ref{2_cycle_rep} (e), \ref{sep_rep} (c) and Equation \ref{Q} tell us that $\rt \circ \sigma_A$ coincide with $\rho$ on $\mathcal{P} \cup \{\beta_1,\ldots,\beta_g,\gamma_1,\ldots,\gamma_{g-1}\}$. As $A$ is a $2p$-th primitive root of unity with $p \ge 3$, $\mathcal{P} \cup \{\beta_1,\ldots,\beta_g,\gamma_1,\ldots,\gamma_{g-1}\}$ generate $S_A(\Sigma)$ by \cite[Thm 1.1]{San} and hence $\rt \circ \sigma_A=\rho$.

\subsection{Classical shadow} \label{classical_shadow}
We still work with an irreducible representation $\rho : S_A(\Sigma) \to \mathrm{End}(V)$ with classical shadow $r$ satisfying the hypothesis of Theorem \ref{thm:lift_rep}. Let $\rt : \mathcal{A}_A(\Gamma)^0 \to \mathrm{End}(V)$ be the lift of $\rho$ built in the previous subsections.
For $\gamma$ a simple a closed curve on $\Sigma$, let $r_{\gamma} = - \mathrm{Tr}(r(\gamma))$.
\begin{proposition} \label{cshadow} 
Let $\gamma \in \{\beta_1,\ldots,\beta_g,\gamma_1,\ldots,\gamma_{g-1}\}$.
\begin{enumerate}
\item[(a)] If $\gamma$ is a one-cycle as in Figure \ref{fig_cycle} (a), then $\rt(Q_e^{2p}) = x_{e}^{2p} \mathrm{Id}_V$ and $$\rt(E_e^p) = \dfrac{r_{\gamma} x_{e}^{-2p} + r_{t_{e}(\gamma)}}{x_{e}^{2p} -  x_{e}^{-2p}} \mathrm{Id}_V$$
\item[(b)] If $\gamma$ is a two-cycle as in Figure \ref{fig_cycle} (b), then $\rt((Q_b Q_c^{ \pm 1})^p) = x_b^p x_{c}^{\pm p} \mathrm{Id}_V$, 
$$\rt((E_b E_c)^p) = \dfrac{r_{\gamma} x_b^{-2p} x_c^{-2p}- r_{t_b(\gamma)} x_c^{-2p}-r_{t_c(\gamma)} x_b^{-2p}+ r_{t_bt_c(\gamma)}}{(x_b^{2p}-x_b^{-2p})(x_c^{2p}-x_c^{-2p})} \mathrm{Id}_V$$
and $$\rt((E_b E_c^{-1})^p)= \dfrac{-r_{\gamma} x_b^{-2p} x_c^{2p}+ r_{t_b(\gamma)} x_c^{2p}+r_{t_c(\gamma)} x_b^{-2p}- r_{t_bt_c(\gamma)}}{(x_b^{2p}-x_b^{-2p})(x_c^{2p}-x_c^{-2p}) \omega} \mathrm{Id}_V$$ where $$\omega = -\prod_{k=0}^{p-1} \dfrac{U\big( (-A)^k x_{a'} x_c x_b^{-1} \big)U\big( (-A)^k x_{a} x_c x_b^{-1} \big)}{U\big( (-A)^k x_c^2 \big)^2}$$
\item[(c)] If $\gamma$ is a separating edge curve as in Figure \ref{fig_cycle} (c), $\rt(Q_c^p) = x_c^p \mathrm{Id}_V$ and
$$\rt(E_c^{2p}) = \dfrac{r_{t_c^{-1}(\gamma)} x_c^{-2p} + r_{\gamma} r_c +r_{t_c(\gamma)} x_c^{2p}}{(x_c^{2p}-x^{-2p})^2 r_c \omega'} \mathrm{Id}_V$$
where $\omega' = \prod_{k=0}^{p-1} U\big((-A)^{k} x_{d_1} x_{d_4} x_{c}^{-1}\big) U\big((-A)^{k} x_{d_2} x_{d_3} x_{c}^{-1}\big)$.

\end{enumerate}

\end{proposition}

\begin{proof}
Before starting the proof, let us recall some important facts from Lemma \ref{1_cycle_rep}. Given $e$ an edge, we can decompose $V$ as the direct sum of the subspaces $V_{e,k} = \mathrm{Ker}\Big(\rho(e)+(x_{e}^2 A^{2k+2} + x_{e}^{-2} A^{-2k-2}) \mathrm{Id}_V \Big).$ 

Moreover, each $V_{e,k}$ is stable by all operators $\rt(Q_e)$ and $\rt(E_f), \rt(Q_f)$ with $f\neq e,$ and finally we have $\rt(E_e)(V_{e,k})\subset V_{e,k+1}.$ When computing $T_p(\gamma)$ for $\gamma$ a curve on $\Sigma$ and where $T_p$ is the $p$-th Chebychev polynomial, since we know that $T_p(\gamma)$ is a multiple of $id_V,$ it is sufficient to compute the "diagonal part" of $T_p(\gamma),$ i.e. the contribution that corresponds to maps $V_{e,k}\longrightarrow V_{e,k},$ since the non-diagonal parts will vanish.

(a) $\rt(Q_e^{2p}) = x_{e}^{2p} \mathrm{Id}_V$ is immediate from the definition of $\rt(Q_e)$. Let us compute $\rt(E_e^{p})$, notice that taking the $p$-th Chebyshev in the equality $\rho(\gamma) = \rt(E_e) + \rt(E_e^{-1}) \rt(F)$ we have $$T_p(\rho(\gamma)) = \rt(E_e^p) + \rt((E_e^{-1}F)^p)$$ Similarly $\rho(t_e(\gamma)) = -A^3 \rt(E_e) \rt( Q_e^2) - A^{-1} \rt(E_e^{-1} F) \rt(Q_e^{-2}) $ gives $T_p(\rho(t_e(\gamma))) =  -\rt(E_e^p) \rt( Q_e^{2p}) - \rt((E_e^{-1} F)^p) \rt(Q_e^{-2p})$. We thus have the following system
$$\left\{
    \begin{array}{llll}
r_{\gamma} \mathrm{Id}_V & = \rt(E_e^p) + \rt((E_e^{-1}F)^p)   \\
r_{t_e(\gamma)} \mathrm{Id}_V & =  -\rt(E_e^p) x_e^{2p} - \rt((E_e^{-1} F)^p) x_e^{-2p}  \end{array}
\right.
$$ The conclusion is obtained by solving this system.

(b) $\rt(Q_b Q_c^{ \pm 1}) = x_b^p x_{c}^{ \pm p} \mathrm{Id}_V$ is immediate. For the other part, it is very similar to (a). This time we start from the system :
$$\left\{
    \begin{array}{llll}
\sigma(\gamma) & = X_{1,1} + X_{1,-1} + X_{-1,1} + X_{-1,-1}   \\
\sigma(t_b(\gamma)) & = -A^3 X_{1,1} Q_b^2 - A^3 X_{1,-1} Q_b^2 - A^{-1} X_{-1,1} Q_b^{-2} - A^{-1} X_{-1,-1} Q_b^{-2} \\
\sigma(t_c(\gamma)) & = -A^3 X_{1,1} Q_c^2 - A^{-1} X_{1,-1} Q_c^{-2} - A^{3} X_{-1,1} Q_c^{2} - A^{-1} X_{-1,-1} Q_c^{-2} \\
\sigma(t_bt_c(\gamma)) & = A^6 X_{1,1} Q_b^2 Q_c^2 + A^2 X_{1,-1} Q_b^2 Q_c^{-2} + A^2 X_{-1,1}Q_b^{-2} Q_c^2 + A^{-2} X_{-1,-1} Q_b^{-2} Q_c^{-2}  \end{array}
\right.
$$ where $X_{\epsilon_1,\epsilon_2} = E_b^{\epsilon_1} E_c^{\epsilon_2} F_{\epsilon_1,\epsilon_2}$ with $F_{1,1} = 1$. Applying $\rt$ and taking the $p$-th Chebyshev polynomial for these four equalities, one gets : 

$$\left\{
    \begin{array}{llll}
r_{\gamma} & = X_{1,1}^p + X_{1,-1}^p + X_{-1,1}^p + X_{-1,-1}^p   \\
r_{t_b(\gamma)} & =  X_{1,1}^p x_b^{2p} + X_{1,-1}^p x_b^{2p} + X_{-1,1}^p x_b^{-2p} + X_{-1,-1}^p x_b^{-2p} \\
r_{t_c(\gamma)} & =  X_{1,1}^p x_c^{2p} + X_{1,-1}^p x_c^{-2p} +  X_{-1,1}^p x_c^{2p} + X_{-1,-1}^p x_c^{-2p} \\
r_{t_bt_c(\gamma)} & =  X_{1,1}^p x_b^{2p} x_c^{2p} + X_{1,-1}^p x_b^{2p} x_c^{-2p} +  X_{-1,1}^p x_b^{-2p} x_c^{2p} + X_{-1,-1}^p x_b^{-2p} x_c^{-2p} \end{array}
\right.
$$
In this system, $X_{\pm 1, \pm 1}$ has to be understood as $\rt(X_{\pm 1, \pm 1})$. As $X_{1,1}^p = \rt((E_b E_c)^p)$, the desired expression is immediate by solving the system. For $\rt((E_b E_c^{-1})^p)$, the resolution of the system gives 
$$X_{1,-1}^p  = \rt((E_b E_c^{-1} F_{1,-1})^p) = \dfrac{-r_{\gamma} x_b^{-2p} x_c^{2p}+ r_{t_b(\gamma)} x_c^{2p}+r_{t_c(\gamma)} x_b^{-2p}- r_{t_bt_c(\gamma)}}{(x_b^{2p}-x_b^{-2p})(x_c^{2p}-x_c^{-2p})} \mathrm{Id}_V$$ Now from Lemma \ref{2_cycle_rep} (a),(b) and the fact that $\rt(Q_a)$, $\rt(Q_{a'})$ commute with $\rt(E_b E_c^{-1})$ we get 
$$\rt((E_b E_c^{-1} F_{1,-1})^p) = - \rt((E_b E_c^{-1})^p)\prod_{k=0}^{p-1} \dfrac{U(A^{-2k} Q_{a'} Q_c Q_b^{-1}) U(A^{-2k} Q_a Q_c Q_b^{-1})}{U(A^{2-2k} Q_c^2) U(A^{-2k} Q_c^2)}$$ Let $v$ an eigenvector commun to $Q_a,Q_{a'},Q_b,Q_c$, we see that $$\dfrac{U(A^{-2k} Q_{a'} Q_c Q_b^{-1}) U(A^{-2k} Q_a Q_c Q_b^{-1})}{U(A^{2-2k} Q_c^2) U(A^{-2k} Q_c^2)} v = \prod_{k=0}^{p-1} \dfrac{U\big( (-A)^k x_{a'} x_c x_b^{-1} \big)U\big( (-A)^k x_{a} x_c x_b^{-1} \big)}{U\big( (-A)^k x_c^2 \big)^2} v$$ Thus $\rt((E_b E_c^{-1} F_{1,-1})^p) = \rt((E_b E_c^{-1})^p) \omega$ and we can conclude.

(c) The equality $\rt(Q_c^p) = x_c^p \mathrm{Id}_V$ is clear. Let us prove the other one. Applying the $p$-th Chebyshev polynomial to $\rho(\gamma)= \rt(E_c^2) \rt(G_2) + \rt(G_0) + \rt(E_c^{-2}) \rt(G_{-2})$ we get $$T_p(\rho(\gamma)) = \rt((E_c^2 G_2)^p)+ H + \rt((E_c^{-2} G_{-2})^p)$$ where $H$ is the degree zero term in $E_c$. We recall that
$$\sigma(t_c(\gamma)) = E_c^{2} G_2 A^8 Q_c^4 + G_0+ E_c^{-2}G_{-2}  Q_c^{-4}$$
We want to compute $T_p(\rho(t_c(\gamma)))$, to do that let us introduce an algebra automorphism $\tau_c$ on $\mathcal{A}_A(\Gamma)$ by the formulas $\tau_c(Q_f) = Q_f$ for all $f \in \mathcal{E}$, $\tau_c(E_f) = E_f$ for $f \neq c$ and $\tau_c(E_c) = (-A)^3 E_c Q_c^2$. This indeed defines an automorphism of $\mathcal{A}_A(\Gamma)$. Note that we have that $\tau_c(E_c^k) = (-A)^{(k+1)^2-1} E_c^k$ (for $k \in \Z$) and $\sigma(t_c(\gamma)) = \tau_c(\sigma(\gamma))$. We deduce that $$T_p(\rho(t_{c}(\gamma))) = \rt(\tau_c(\sigma(\gamma))) = \rt((E_c^{2}G_2  A^{8} Q_c^{4})^p) + H+ \rt((E_c^{-2} G_{-2}  Q_c^{-4} )^p)$$ Similarly 
$$T_p(\rho(t_c^{-1}(\gamma))) = \rt(\tau_c^{-1}(\sigma(\gamma))) = \rt((E_c^{2}G_2  A^{-8} Q_c^{-4})^p) + H+ \rt((E_c^{-2} G_{-2}  Q_c^{4})^p) $$ Now we have $$(E_c^{2}G_2  A^{\pm 8} Q_c^{\pm 4})^p = (E_c^2 G_2)^p x_c^{ \pm 4p} \quad \text{and} \quad  (E_c^{-2} G_{-2}  Q_c^{\pm 4})^p = (E_c^{-2} G_{-2}  )^p x_c^{\pm 4p}$$ We deduce the following system 
$$\left\{
    \begin{array}{llll}
T_p(\rho(t_c^{-1}(\gamma))) & = \rt((E_c^{2}G_2)^p)  x_c^{-4p} + H+ \rt((E_c^{-2} G_{-2})^p  x_c^{4p})   \\
T_p(\rho(\gamma)) & = \rt((E_c^2 G_2)^p)+ H+ \rt((E_c^{-2} G_{-2})^p) \\
T_p(\rho(t_c(\gamma))) & = \rt((E_c^{2} G_2)^p) x_c^{4p} + H+ \rt((E_c^{-2}G_{-2})^p)  x_c^{-4p}  \end{array}
\right.
$$ solving this system gives 
$$\rt((E_c^2 G_2)^p) = \dfrac{T_p(t_c^{-1}(\gamma)) x_c^{-2p} - T_p(\gamma) (x_c^{2p}+x_c^{-2p})+T_p(t_c(\gamma)) x_c^{2p}}{(x_c^{2p}-x^{-2p})^2 (x_c^{2p}+x_c^{-2p})} \mathrm{Id}_V$$
On the other hand, 

$$\rt((E_c^2 G_2)^p) = -\rt(E_c^{2p}) \prod_{k=0}^{p-1} U((-A)^{k} x_{d_1} x_{d_4} x_{c}^{-1}) U((-A)^{k} x_{d_2} x_{d_3} x_{c}^{-1})$$ 
\end{proof}

\begin{proof}[Proof of Corollary \ref{cor:unicity}] Let $\rho_1$ and $\rho_2$ be two irreducible representations of $S_A(\Sigma)$ with the same classical shadow $r$ satisfying the hypothesis of Theorem \ref{thm:lift_rep}. Let $\rt_1$ and $\rt_2$ the lifts of $\rho_1$ and $\rho_2$ to $\mathcal{A}_A(\Gamma)^0$ built from the previous subsections. We build the lifts $\rt_1$ and $\rt_2$ from the same quantities $\{x_{\alpha} \, , \, \alpha \in \mathcal{P} \}$. Proposition \ref{cshadow} shows that $\rt_1$ and $\rt_2$ have the same scalar values on the $p$-th powers of the non central generators given in Lemma \ref{gen_t}. Finally Proposition \ref{rep_Q_t} allows us to conclude.

\end{proof}
\appendix
\section{Skein computations}
\label{sec:skein_computations}
In this section, we will give more details on some of the skein computations skipped in Section \ref{sec:com1cse} and \ref{sec:com2cse}, showing that the skein elements $C_1$ in the proof of Lemma \ref{com1cse} and $C_{0,0,0 }$ in the proof of Lemma \ref{com2cse} both vanish.

Recall that for $x,y\in S(\Sigma),$ we write $[x,y]_A$ for $Axy-A^{-1}yx.$

\begin{lemma} \label{lemma:intersect1} Let $a$ and $b$ be two simple closed curves on $\Sigma,$ viewed as elements of $S(\Sigma),$ that intersect once geometrically. Let $t_a$ and $t_b$ be the associated Dehn twists. Then:
\begin{itemize}
\item[-]$a b=A t_a(b)+A^{-1}t_a^{-1}(b)$
\item[-]$[a,b]_A=(A^2-A^{-2})t_a(b)$
\item[-]$t_a(b)=t_b^{-1}(a).$
\end{itemize}
\end{lemma}
\begin{proof}
The first two points are direct consequence of Kauffman relations, and the third is a simple isotopy.
\end{proof}
\begin{lemma}\label{lemma:skein_comput1} Let $\gamma, \tau,\bar{\tau},\beta,c,e$ be the curves described in Section \ref{sec:com1cse}, let $\varphi=\gamma-t_c(\gamma),$ $\psi= \bar{\tau}-\tau=\bar{\tau}-t_c(\tau),$ and let $$C_1=-A^5[t_e(\beta),\psi]_A-A^3[\varphi,t_e(\beta)]_A+A^4[\varphi,\beta]_A e.$$
Then $C_1=0.$
\end{lemma}
\begin{proof}
One can check that $C_1=x_1-t_c(x_1),$ where $x_1=-A^5[t_e(\beta),\bar{\tau}]_A-A^3[\gamma,t_e(\beta)]_A+A^4[\gamma,\beta]_A.$ We compute $x_1$ using Lemma \ref{lemma:intersect1}:
\begin{multline*}x_1=(A^2-A^{-2})\left(-A^5 t_e t_{\beta} t_c^{-1}(\tau) -A^3 t_{t_e(\beta)}^{-1}(\gamma) +A^4 t_{\beta}^{-1}(\gamma)e \right)
\\ =(A^2-A^{-2})\left( -A^5 t_e t_{\beta} t_c^{-1}(\tau) -A^3 t_{t_e(\beta)}^{-1}(\gamma) +A^5 t_e^{-1}t_{\beta}^{-1}(\gamma) +A^3 t_e t_{\beta}^{-1}(\gamma) \right)
\end{multline*}
To conclude, let us remark that $t_{t_e(\beta)}^{-1}(\gamma)=t_e t_{\beta}^{-1} t_e^{-1}(\gamma)= t_e t_{\beta}^{-1}(\gamma),$ and that a simple isotopy shows that $t_e^{-1}t_{\beta}^{-1}(\gamma)\sim t_e t_{\beta} t_c^{-1}(\tau).$ Hence $x_1=0$ and therefore $C_1=0.$
\end{proof}
\begin{lemma}\label{lemma:skein_comput2} Let $\gamma, \bar{\tau},\beta,c,d_1,d_4$ be the curves described in Section \ref{sec:com2cse}, let $\varphi=\gamma-t_c(\gamma)$ and $\psi=\bar{\tau}-t_c(\bar{\tau})$ and let
\begin{multline*}C_{0,0,0}=-[\psi,t_{d_4}(\beta)]_A+A^{-1}[\psi,t_{d_4}t_{d_1}(\beta)]_A+A^{-1}[\beta,\varphi]_A
\\-A[\beta,\psi]_A c +A[\psi,\beta]_A d_1 d_4 -[\psi,t_{d_1}(\beta)]_A d_4.
\end{multline*}
Then $C_{0,0,0}=0.$
\end{lemma}
\begin{proof}
Again, we notice that $C_{0,0,0}=x_{0,0,0}-t_c(x_{0,0,0})$ where 
\begin{multline*}x_{0,0,0}=-[\bar{\tau},t_{d_4}(\beta)]_A+A^{-1}[\bar{\tau},t_{d_4}t_{d_1}(\beta)]_A+A^{-1}[\beta,\gamma]_A
\\ -A[\beta,\bar{\tau}]_A c +A[\bar{\tau},\beta]_A d_1 d_4 -[\bar{\tau},t_{d_1}(\beta)]_A d_4.
\end{multline*}
 Using Lemma \ref{lemma:intersect1}, we get:
\begin{multline*}\frac{1}{A^2-A^{-2}}x_{0,0,0}=-t_{d_4}t_{\beta}^{-1}(\bar{\tau})d_1+A^{-1}t_{d_1}t_{d_4}t_{\beta}^{-1}(\bar{\tau})+A^{-1}t_{\beta}(\gamma)
\\ +A t_{\beta}(\bar{\tau})c+A t_{\beta}^{-1}(\bar{\tau}) d_1 d_4 -t_{d_1}t_{\beta}^{-1}(\bar{\tau})d_4.
\end{multline*}
Notice that $t_{\beta}(\bar{\tau})c=t_{\beta}(\bar{\tau}c).$ The curves $\bar{\tau}$ and $c$ have geometric intersection $2$ and an easy skein computation shows that $\bar{\tau}c=A^2 t_c^{-1}(\gamma) + \delta_2 +A^{-2} \gamma,$ where $\delta_2=d_1d_2+d_3d_4.$ Using this and the second part of Lemma \ref{lemma:intersect1}, we get
\begin{multline*}\frac{1}{A^2-A^{-2}}x_{0,0,0}=-A^{-1} t_{d_1}t_{d_4} t_{\beta}^{-1}(\bar{\tau}) -A t_{d_1}^{-1}t_{d_4}t_{\beta}^{-1}(\bar{\tau}) +A^{-1}t_{d_1}t_{d_4}t_{\beta}^{-1}(\bar{\tau})+A^{-1}t_{\beta}(\gamma)
\\-A^3 t_{\beta}t_c^{-1}(\gamma) -At_{\beta}(\delta_2)-A^{-1}t_{\beta}(\gamma)
\\ +A^3 t_{d_1}^{-1}t_{d_4}^{-1}t_{\beta}^{-1}(\bar{\tau})+A t_{d_1}t_{d_4}^{-1}t_{\beta}^{-1}(\bar{\tau}) + A  t_{d_1}^{-1}t_{d_4}t_{\beta}^{-1}(\bar{\tau}) +A^{-1}  t_{d_1}t_{d_4}t_{\beta}^{-1}(\bar{\tau})
\\ -A t_{d_1}t_{d_4}^{-1}t_{\beta}^{-1}(\bar{\tau})-A^{-1} t_{d_1} t_{d_4} t_{\beta}^{-1}(\bar{\tau})
\\ = -A t_{\beta}(\delta_2)-A^3 t_{\beta}t_c^{-1}(\gamma) +A^3 t_{d_1}^{-1} t_{d_4}^{-1} t_{\beta}^{-1}(\bar{\tau}).
\end{multline*}
A drawing shows that the simple closed curves $t_{d_1}^{-1} t_{d_4}^{-1} t_{\beta}^{-1}(\bar{\tau})$ and $t_{\beta}t_c^{-1}(\gamma)$ are  actually isotopic. Hence $x_{0,0,0}=-A(A^2-A^{-2})t_{\beta}(\delta_2).$ But $t_{\beta}(\delta_2)$ is a linear combination of multicurves that are disjoint from $c,$ hence $x_{0,0,0}$ is $t_c$ invariant, and $C_{0,0,0}=0.$
\end{proof}
\bibliographystyle{hamsalpha}
\bibliography{biblio}
\end{document}